\documentclass[aop,MSNbibl,seceqn,dvips]{arximspdf}
\usepackage{mathbh,esint}

\doi{10.1214/13-AOP833} %kopijuoti is PTS
\volume{42}
\issue{6}
\pubyear{2014}
\firstpage{2558}
\lastpage{2594}

\makeatletter
\newcommand{\rrvert}{\vert}
\newcommand{\llvert}{\vert}
\newtheorem{teo}{Theorem}
\newtheorem{cor}{Corollary}[section]
\newtheorem{lem}[cor]{Lemma}
\newtheorem{prop}[cor]{Proposition}
\newproclaim{remark}[cor]{Remark}
\newcommand{\tr}{\operatorname{tr}}
\newcommand{\USC}{\operatorname{USC}}
\newcommand{\LSC}{\operatorname{LSC}}
\newcommand{\dist}{\operatorname{dist}}
\newcommand{\esssup}{\mathop{\operatorname{ess}\operatorname{sup}}}
\newcommand{\essinf}{\mathop{\operatorname{ess}\operatorname{inf}}}
\makeatother

\begin{document}
\begin{frontmatter}

\title{Regularity and stochastic homogenization of fully nonlinear
equations without uniform ellipticity}
\runtitle{Regularity and stochastic homogenization}

\begin{aug}
\author[A]{\fnms{Scott N.} \snm{Armstrong}\corref{}\ead[label=e1]{armstrong@ceremade.dauphine.fr}\thanksref{T1}}
\and
\author[B]{\fnms{Charles K.} \snm{Smart}\ead[label=e2]{smart@math.mit.edu}\thanksref{T2}}
\runauthor{S. N. Armstrong and C. K. Smart}
\affiliation{University of Wisconsin and Universit\'e
Paris-Dauphine,\\ and Massachusetts Institute of Technology}
\address[A]{Department of Mathematics\\
University of Wisconsin\\
Madison, Wisconsin 53706\\
USA\\
and\\
Ceremade (UMR CNRS 7534)\\
Universit\'e Paris-Dauphine\\
%Place du Mar\'echal de Lattre de Tassigny\\
75775 Paris Cedex 16\\
France\\
\printead{e1}} %adresu isvedimo komanda gale!
\address[B]{Department of Mathematics\\
Massachusetts Institute of Technology\\
Cambridge, Massachusetts 02139\\
USA\\
\printead{e2}}
\thankstext{T1}{Supported in part by NSF Grant DMS-10-04645 and by a Chaire Junior of la Fondation Sciences Math\'ematiques de Paris.}
\thankstext{T2}{Supported in part by NSF Grant DMS-10-04595.}
\end{aug}

% HISTORY:
\received{\smonth{9} \syear{2012}}

% ABSTRACT
%
\begin{abstract}
We prove regularity and stochastic homogenization results for certain
degenerate elliptic equations in nondivergence form. The equation is
required to be strictly elliptic, but the ellipticity may oscillate on
the microscopic scale and is only assumed to have a finite $d$th
moment, where $d$ is the dimension. In the general stationary-ergodic
framework, we show that the equation homogenizes to a deterministic,
uniformly elliptic equation, and we obtain an explicit estimate of the
effective ellipticity, which is new even in the uniformly elliptic
context. Showing that such an equation behaves like a uniformly
elliptic equation requires a novel reworking of the regularity theory.
We prove deterministic estimates depending on averaged quantities
involving the distribution of the ellipticity, which are controlled in
the macroscopic limit by the ergodic theorem. We show that the moment
condition is sharp by giving an explicit example of an equation whose
ellipticity has a finite $p$th moment, for every $p< d$, but for which
regularity and homogenization break down. In probabilistic terms, the
homogenization results correspond to quenched invariance principles for
diffusion processes in random media, including linear diffusions as
well as diffusions controlled by one controller or two competing players.
\end{abstract}

% KEYWORDS
% Pirmas kwd is didziosios raides
%

\end{frontmatter}

\setcounter{footnote}{2}

%s1 #&#
\section{Introduction} \label{I}
We prove stochastic homogenization and regularity estimates for fully
nonlinear elliptic equations in nondivergence form without the
assumption of uniform ellipticity. The equations we consider are
strictly elliptic but may have ellipticities which are arbitrarily
large and oscillating on the microscopic scale. We derive new
(deterministic) regularity estimates and show that, under the
assumption that the $d$th moment of the ellipticity is finite, such a
degenerate equation homogenizes in the macroscopic limit to an
effective equation which is uniformly elliptic. Our analysis yields an
explicit estimate for the effective ellipticity which is new, to our
knowledge, even in the linear, uniformly elliptic setting. In terms of
probability, the main homogenization result, in the special case of
linear equations, is equivalent to a quenched invariance principle for
a diffusion in a random environment. For a nonlinear, positively
homogeneous equation, the homogenization result gives similar
information about the quenched behavior of controlled diffusions in
random environments.

The simplest example of interest is the linear equation
%e1.1 #&#
\begin{equation}
\label{lin} -\sum_{i,j=1}^da_{ij}
\biggl(\frac{x}\varepsilon,\omega \biggr) u^\varepsilon_{x_ix_j}
= f\qquad\mbox{in } U \subseteq\mathbb R^d,  d\geq1.
\end{equation}
The coefficient matrix $(a_{ij})$, which depends on the random
parameter $\omega$, called \textit{the environment}, is assumed to be
stationary-ergodic and to satisfy the ellipticity condition
%e1.2 #&#
\begin{equation}
\label{linellip} \lambda(\omega) |\xi|^2 \leq\sum
_{i,j=1}^da_{ij}(0,\omega) \xi
_i\xi_j \leq\Lambda|\xi|^2 \qquad\mbox{for all } \xi\in\mathbb R^d,
\end{equation}
where $\Lambda> 0$ is a given constant and $\lambda=\lambda(\omega
)$ is a nonnegative random variable. Papanicolaou and Varadhan~\cite
{PV1,PV2} proved that, if the equation is uniformly elliptic, that is,
$\lambda(\omega) \geq\lambda_0 >0$, then in the almost sure
asymptotic limit $\varepsilon\to0$, equation~(\ref{lin}) is governed
by an effective equation of the form
\[
-\sum_{i,j=1}^d\bar a_{ij}
u_{x_ix_j} = f,
\]
where the coefficient matrix $(\bar a_{ij})$ is uniformly elliptic
with ellipticity $\Lambda/\lambda_0$, that is, for every $\xi\in
\mathbb R^d$,
%e1.3 #&#
\begin{equation}
\label{effell} \lambda_0|\xi|^2 \leq\sum
_{i,j=1}^d\bar a_{ij}
\xi_i\xi_j \leq\Lambda|\xi|^2.
\end{equation}

In this paper, we extend this result to the case that $\lambda> 0$ and
$\inf\lambda=0$ under the assumption that $\mathbb{E}  [
\lambda^{-d} ] < \infty$. Due to the presence of small pockets
of arbitrarily large ellipticity which become dense for small
$\varepsilon$, it is by no means clear at first glance that~(\ref
{lin}) should behave, in the macroscopic limit, like a uniformly
elliptic equation. Our result demonstrates that this finite moment
condition on $\lambda^{-1}$ ensures that these tiny regions of very
large ellipticity may be effectively controlled and that~(\ref{lin})
becomes a uniformly elliptic equation in the limit $\varepsilon\to0$.
Furthermore, we show that the moment condition is sharp by exhibiting
an explicit example, for each $p <d$ in arbitrary dimension $d\geq1$,
of a nonlinear equation with a finite $p$th moment and having a finite
range of dependence, for which homogenization fails.

Previously, in this linear setting, a result similar to our
homogenization result has been proved in a recent paper of Guo and
Zeitouni~\cite{GZ} using probabilistic methods. They give a quenched
invariance principle for random walks in a random environment, which
has an equivalent formulation in terms of stochastic homogenization of
a discrete equation (on the lattice $\mathbb{Z}^d$ rather than
$\mathbb R^d$). They require a~slightly stronger moment condition,
namely that $\lambda^{-1}$ have a finite $p$th moment for some $p> d$.
That homogenization occurs in the case $p=d$ is new here.

While our results are therefore of interest in the linear setting, we
analyze much more general fully nonlinear equations of the form
%e1.4 #&#
\begin{equation}
\label{pde} F \biggl(D^2u^\varepsilon,\frac{x}\varepsilon,
\omega \biggr) = 0.
\end{equation}
Such nonlinear equations are more difficult to analyze than~(\ref{lin}), due to the absence of invariant measures---the key tool used
in~\cite{PV1,PV2} (and~\cite{GZ}) to overcome the problem of the
``lack of compactness.'' The results here are the first for such
equations without a uniform ellipticity assumption. Previously, the
homogenization of fully nonlinear equations in stationary-ergodic media
was proved in the uniformly elliptic case by Caffarelli, Souganidis and
Wang~\cite{CSW}. They introduced a new method based on the obstacle
problem, a strategy which we also use in this paper.

The problem one encounters when trying to homogenize~(\ref{pde})
outside of the uniformly elliptic regime is that most of the regularity
theory needed to implement the method of~\cite{CSW} is destroyed by
(even tiny) regions of high ellipticity. It therefore seems hopeless,
at first glance, to implement the techniques of~\cite{CSW}, since they
make heavy use of the regularity tools.

To overcome this difficulty, prove new (deterministic) regularity
estimates in which the dependence on a uniform upper bound for the
ellipticity is replaced by that of its~$L^d$-norm. In particular, we
prove a decay of the oscillation lemma at unit scale. This result, and
the new arguments we introduce to obtain it, are of independent
interest. Indeed, in sharp contrast to the situation for divergence
form equations, there are very few results in the literature for
equations in nondivergence form, which provide estimates of solutions
in terms of averaged quantities.

The estimates refine the classical regularity theory~\cite{CC} and
require several new ideas. One of the basic techniques involves using
the area formula to estimate the size of certain ``contact sets''
between supersolutions and certain families of smooth test functions.
This method is a generalization of the classical ABP inequality and was
previously used by Cabr{\'e}~\cite{C} to obtain the Harnack inequality
on Riemannian manifolds with nonnegative sectional curvature, by Savin
in his proof of De Giorgi's conjecture~\cite{S2} and in his beautiful proof of
the Harnack inequality in~\cite{S}. In each of these works,
supersolutions are touched from below not only by planes, but by
translations of balls and paraboloids. In the present paper, one of the
key arguments involves touching from below by translations of the \textit{singular function} $|x|^{-\alpha}$, for suitably large $\alpha> 0$.

Since the~$L^d$ norm of the ellipticity is controlled on the
macroscopic scale almost surely by the ergodic theorem, the regularity
results provide effective control on the solutions of~(\ref{pde}) for
small~$\varepsilon$. This allows us to homogenize the equation by
suitably adapting the arguments of~\cite{CSW}. We expect this two-step
approach to homogenization, in which one obtains ``effective''
regularity and then uses this to homogenize the equation, to be useful
in other situations.

As a byproduct of our analysis, we obtain an estimate for the effective
ellipticity which is new even in the uniformly elliptic case, which
states that~(\ref{effell}) holds for a~$\lambda_0>0$ which, in
addition to $\Lambda$ and $d$, depends only on $\mathbb{E}  [
\lambda^{-d}  ]$. It is, to our knowledge, the first such bound
for the homogenized coefficients of nondivergence form equations which
is nontrivial in the sense that it is given in terms on the averaged
microscopic behavior of the equation rather than its uniform properties.

As mentioned above, the homogenization result has an equivalent
probabilistic formulation, at least in the linear case, as a quenched
invariance principle for the corresponding diffusion in the random
environment. It also provides information regarding the recurrence or
transience of the diffusion (see~\cite{GZ}). If $F$ is nonlinear but
positively homogeneous, the fully nonlinear equation~(\ref{pde}) is a
Bellman--Isaacs equation which arises in the theory of stochastic
optimal control and two-player stochastic differential games, and the
homogenization result yields similar information about these more
general diffusion processes. Although we do not explore this point
here, we remark that the recurrence verses transience of such
controlled diffusion processes in an isotropic environment was
characterized in~\cite{ASS}, and this result applied to the effective
operator $\overline F$, together with its proof, gives information
about the corresponding questions for controlled diffusions in random
environments.

We now give the precise statement of our results, beginning with the
modeling assumptions.

\subsection*{The model}
We work in Euclidean space $\mathbb R^d$ in dimension $d\geq1$. The
random environment is modeled by a probability space $(\Omega,\mathcal
{F},\mathbb{P})$ endowed with an ergodic group $\tau=  ( \tau_y
)_{y\in\mathbb R^d}$ of $\mathcal{F}$-measurable, $\mathbb
{P}$-preserving transformations on $\Omega$. That is, the action $\tau
$ of $\mathbb R^d$ on $\Omega$ satisfies
%e1.5 #&#
\begin{equation}
\label{Ppres} \mathbb{P}[A] =\mathbb{P}[\tau_y A] \qquad\mbox{for
every }y\in\mathbb R^d, A\in\mathcal{F}
\end{equation}
and, for every $A\in\mathcal{F}$,
%e1.6 #&#
\begin{equation}
\label{erg}
\tau_y A = A\qquad\mbox{for every } y\in\mathbb
R^d\qquad\mbox{implies that }\mathbb{P}[A] = 0\quad\mbox{or}\quad \mathbb
{P}[A] =1.\hspace*{-25pt}
\end{equation}
The nonlinear elliptic operator is a map $F\dvtx {\mathbb{S}^d}\times
\mathbb R^d\times\Omega\to\mathbb{R}$ (here ${\mathbb{S}^d}$
denotes the space of $d$-by-$d$ symmetric matrices) which satisfies
each of the following four conditions:

\begin{longlist}[(F3)]
\item[(F1)] \textit{Stationarity}: for every $M\in{\mathbb{S}^d}$,
$y,z\in\mathbb R^d$ and $\omega\in\Omega$,
\[
F(M,y,\tau_z\omega) = F(M,y+z,\omega).
\]

\item[(F2)] \textit{Local uniform ellipticity}: there exists a constant
$\Lambda\geq1$ and a nonnegative random variable $\lambda\dvtx \Omega\to
[0,\Lambda]$ such that $\mathbb{P}[\lambda>0 ] = 1$ and, for every
\mbox{$M,N\in{\mathbb{S}^d}$,} $\omega\in\Omega$ and $y\in B_1$,
\[
{\mathcal{P}}^-_{\lambda(\omega),\Lambda}(M-N) \leq F(M,y,\omega) - F(N,y,\omega) \leq{
\mathcal{P}}^+_{\lambda(\omega),\Lambda} (M-N).
\]
(Here, ${\mathcal{P}}^\pm$ are the usual Pucci extremal operators;
see the next section.)

\item[(F3)] \textit{Uniform continuity and boundedness}: for each $R> 0$,
\[
\bigl\{ F(\cdot,\cdot,\omega) \dvtx  \omega\in\Omega \bigr\} \qquad\mbox{is uniformly
equicontinuous on } B_R \times\mathbb R^d
\]
and
\[
\esssup_{\omega\in\Omega} \bigl\llvert F(0,0,\omega) \bigr\rrvert < +\infty.
\]
Moreover, there exists a modulus $\rho\dvtx [0,\infty) \to[0,\infty)$
and a constant $\sigma> \frac{1}2$ such that, for all $(M,p,\omega)\in
{\mathbb{S}^d}\times\mathbb R^d\times\Omega$ and $y,z\in\mathbb R^d$,
\[
\bigl\llvert F(M,y,\omega) - F(M,z,\omega) \bigr\rrvert \leq\rho \bigl(
\bigl(1+|M|\bigr)|y-z|^\sigma \bigr).
\]

\item[(F4)] \textit{Bounded moment of the ellipticity}: the random
variable $\lambda$ satisfies
\[
\mathbb{E} \bigl[ \lambda^{-d} \bigr] < +\infty.
\]
\end{longlist}

\subsection*{The main result}

We now present the homogenization result, which for simplicity we state
in terms of the Dirichlet problem
%e1.7 #&#
\begin{equation}
\label{bvp} \cases{ \displaystyle F \biggl(D^2u^\varepsilon,
\frac{x}\varepsilon,\omega \biggr) = 0, &\quad in $U$,
\vspace*{3pt}\cr
u^\varepsilon= g, &\quad on $\partial U$.}
\end{equation}
Here, $U \subseteq\mathbb R^d$ is a bounded Lipschitz domain and $g\in
C(\partial U)$, and the PDE is to be understood in the viscosity sense
(cf.~\cite{CIL,CC}). By modifying our argument in a very minor way
(only small changes in part three of Section~\ref{HO}), we may
homogenize any other well-posed problem involving $F$, including
parabolic equations like
\[
u_t + F \biggl(D^2u,\frac{x}\varepsilon,\omega
\biggr) = 0
\]
with appropriate boundary/initial conditions.

Note that by (F1) and (F2), for each $\varepsilon> 0$, equation~(\ref
{bvp}) is uniformly elliptic with probability one. Indeed, if we take
$\{ B(x_j,1)  \dvtx   1\leq j \leq k \}$ to be a finite covering of
$\varepsilon^{-1} U$, then its ellipticity is bounded by the random variable
\[
\Lambda\sup_{1\leq j \leq k} \lambda^{-1} (\tau _{x_j/\varepsilon}
\omega ),
\]
which is almost surely finite by (F2). See~(\ref{uez}) below. As a
consequence,~(\ref{bvp}) is well-posed and has a unique viscosity
solution $u^\varepsilon=u^\varepsilon(x,\omega)$ belonging to
$C(\overline U)$.

The main homogenization result is the following theorem.

%
%th1 #&#
\begin{teo}\label{H}
Assume \textup{(F1)}, \textup{(F2)}, \textup{(F3)} and \textup{(F4)}. Then there exists an event $\Omega
_1\in\mathcal{F}$ of full probability, a positive constant $0 <
\lambda_0< \Lambda$ which depends only on~$d$, $\Lambda$ and
$\mathbb{E}  [ \lambda^{-d}  ]$ and a function $\overline
F\dvtx {\mathbb{S}^d}\to\mathbb{R}$ which satisfies
\[
{\mathcal{P}}^-_{\lambda_0,\Lambda}(M-N) \leq\overline F(M) - \overline F(N) \leq{
\mathcal{P}}^+_{\lambda_0,\Lambda} (M-N)
\]
such that, for every $\omega\in\Omega_1$, every bounded Lipschitz
domain $U\subseteq\mathbb R^d$ and each $g\in C(\partial U)$, the
unique solution $u^\varepsilon=u^\varepsilon(x,\omega)$ of the
boundary value problem~(\ref{bvp}) satisfies
\[
\lim_{\varepsilon\to0} \sup_{x\in U} \bigl\llvert
u^\varepsilon (x,\omega) - u (x) \bigr\rrvert = 0,
\]
where $u\in C(\overline U)$ is the unique solution of the Dirichlet problem
%e1.8 #&#
\begin{equation}
\label{Ee} \cases{ \overline F \bigl(D^2u \bigr) = 0, &\quad in
$U$,
\vspace*{3pt}\cr
u = g, &\quad on $\partial U$.}
\end{equation}
\end{teo}

\subsection*{A brief literature review}
The modern regularity theory for elliptic equations in nondivergence
form began in the 1980s with the groundbreaking work of Krylov and
Safonov, Evans, Caffarelli and others, and we refer to~\cite{CC} and
the references there for more. For degenerate equations, we are unaware
of much work that can be compared to ours here. An exception is the
\textit{linearized Monge--Amp\`ere equation} which, although degenerate,
possesses a special geometric structure allowing for the development of
a regularity theory, as discovered by Caffarelli and Guti{\'
e}rrez~\cite{CG} (see also Guti{\'e}rrez and Nguyen~\cite{GN} and the
references therein). Recently, there has been some progress in
obtaining Harnack inequalities and H\"older regularity for certain
nonlinear degenerate equations (see, e.g.,~\cite{I,DFQ,IS}). In these
works, the degeneracy of the equation is typically compensated in some
way by dependence on the gradient. A typical model equation considered is
%e1.9 #&#
\begin{equation}
\label{other} |Du|^\gamma F \bigl(D^2u \bigr) = 0,
\end{equation}
where $\gamma> 0$ and $F$ is uniformly elliptic. A solution $u$
of~(\ref{other}) may only be irregular if $|Du|$ is small, and this
allows to compensate for the degeneracy. This is a~very different
situation from the ``naked'' degeneracy of the equations considered here.

The homogenization of linear uniformly elliptic equations in random
media originated in the work of~Papanicolaou and Varadhan~\cite
{PV1,PV2} and~Kozlov \mbox{\cite{K1,K2}}. Later, Dal Maso and Modica~\cite
{DM1,DM2} obtained stochastic homogenization results for nonlinear
equations in divergence form and convex variational problems. The
homogenization of uniformly elliptic, nonlinear equations in
nondivergence form was first considered in the periodic setting by
Evans~\cite{E2} and much later by Caffarelli, Souganidis and
Wang~\cite{CSW} in random media. In contrast to the divergence form
case (cf.~\cite{CK}), little seems to be known about the homogenized
coefficients for nondivergence form equations, even in the periodic
case, other than what is inherited from the uniform properties of the
medium. As far as quantitative homogenization results, we mention the
work of Yurinski{\u\i}~\cite{Y} and Gloria and Otto \mbox{\cite{GO1,GO2}}
for linear equations and Caffarelli and Souganidis~\cite{CS} for fully
nonlinear equations.

\subsection*{Outline of the paper}
In the next section, we give some preliminary results and notation
needed later in the paper and make some comments about our assumptions.
In Section~\ref{EO}, we develop the deterministic regularity theory.
The proof of~Theorem~\ref{H} is then given in Section~\ref{HO}.
Finally, in Section~\ref{CE} we construct an explicit example to show
that the moment condition (F4) is sharp for general nonlinear equations.

%s2 #&#
\section{Preliminaries} \label{PO}
In this section, we present some background results needed in the rest
of the paper, including the statements of the ergodic theorems we cite,
some remarks about our model and some general remarks concerning
viscosity solutions and semiconcave functions. We begin by reviewing
the notation.

\subsection*{Notation}
The symbols $C$ and $c$ denote positive constants which may vary at
each occurrence and which typically depend on known quantities. We work
in Euclidean space $\mathbb R^d$ for $d\geq1$. We denote the set of
natural numbers by $\mathbb{N}:=\{ 0, 1,\ldots\}$ and $\mathbb{Q}$
is the set of rational numbers. If $r\in\mathbb{R}$, then $\lceil r
\rceil$ denotes the smallest positive natural number which is greater
than or equal to $r$, and we write $\lfloor r \rfloor:= - \lceil-r
\rceil$. The family of bounded Lipschitz subsets of~$\mathbb R^d$ is
denoted by~$\mathcal{L}$. The open ball centered at~$y\in\mathbb R^d$
with radius $r>0$ is \mbox{$B_r(y):=\{ x\in\mathbb R^d  \dvtx    |x-y| < r \}$}
and we write $B_r:=B_r(0)$. If $E\subseteq\mathbb R^d$ is a bounded
Borel set, then $\overline E$ is its closure and $|E|$ is the Lebesgue
measure of $E$. If $f\in L^1(E)$, then the average of $f$ in $E$ is
$\fint_E f(x) \,dx:= |E|^{-1} \int_E f(x)  \,dx$. If $f\dvtx E \to\mathbb
{R}$, then we denote $\operatorname{osc}_E f:= \sup_E f - \inf_E
f$. The characteristic function of a Borel set~$E$ is~$\chi_E$. We
work with a probability space $(\Omega,\mathcal{F},\mathbb{P})$ as
described in the previous section. The indicator random variable of an
event $A\in\mathcal{F}$ is written~$\mathbh{1}_{A}$. We say that
$A\in\mathcal{F}$ is of \textit{full probability} if $\mathbb
{P}[A]=1$. The space of symmetric $d$-by-$d$ matrices is ${\mathbb
{S}^d}$. If $M,N\in{\mathbb{S}^d}$, we write $M \geq N$ if the
eigenvalues of $M-N$ are nonnegative. If $x,y\in\mathbb R^d$, then
$x\otimes y$ denotes the $d$-by-$d$ matrix with entries $(x_iy_j)$. The
trace of $M\in{\mathbb{S}^d}$ is $\tr(M)$. Recall that any $M\in
{\mathbb{S}^d}$ can be uniquely expressed as a difference $M=M_+ -
M_-$ where $M_+M_- = 0$ and $M_+,M_- \geq0$. In particular, if $r\in
\mathbb{R}$, then we write $r_+:= \max\{ 0,r\}$ and $r_-:= (-r)_+$.
The Pucci extremal operators ${\mathcal{P}}^{\pm}$ are defined for $0
< \mu\leq\Lambda$ and $M\in{\mathbb{S}^d}$ by
\[
\mathcal{P}^+_{\mu,\Lambda}(M):= -\mu\tr(M_+) + \Lambda\tr (M_-)
\]
\mbox{and}
\[
{\mathcal{P}}^-_{\mu,\Lambda}(M):= -\Lambda\tr(M_+) + \mu
\tr(M_-).
\]
The elementary properties of the Pucci operators can be found in~\cite
{CC}. Here, we remark only that they are uniformly elliptic, ${\mathcal
{P}}^+_{\mu,\Lambda}$ is convex and ${\mathcal{P}}^-_{\mu,\Lambda
}$ is concave. The set of upper and lower semicontinuous functions on
$V\subseteq\mathbb R^d$ are denoted by $\USC(V)$ and $\LSC(V)$, respectively.

\subsection*{Brief remarks concerning the assumptions}
Note that the restriction of (F2) to $y\in B_1$ is merely for
convenience, it may be extended to all $y\in\mathbb R^d$ by
stationarity. Indeed, the combination of (F1) and (F2) yields, for all
$y,z\in\mathbb R^d$ with $|y-z|< 1$,
%e2.1 #&#
\begin{equation}
\label{uez} \quad {\mathcal{P}}^-_{\lambda(\tau_z\omega),\Lambda}(M-N) \leq F(M,y,\omega) - F(N,y,
\omega) \leq{\mathcal{P}}^+_{\lambda(\tau
_z\omega),\Lambda} (M-N).
\end{equation}
In light of~(\ref{uez}), it is convenient to abuse notation by writing
$\lambda(z,\omega)=\lambda(\tau_z\omega)$. Note also that due
to~(F3) we may suppose that
%e2.2 #&#
\begin{equation}
\label{lambec} \bigl\{ \lambda(\cdot,\omega) \dvtx  \omega\in\Omega \bigr\}
\qquad\mbox{is uniformly equicontinuous on } \mathbb R^d.
\end{equation}
Otherwise, we simply redefine $\lambda$ to be the largest quantity
which satisfies (F2), which then satisfies~(\ref{lambec}) by (F3). The
operators on the leftmost and rightmost side of~(\ref{uez}) are the
minimal and maximal operators, respectively, which satisfy conditions
(F1)--(F3). In particular, since $\lambda(\cdot,\omega)>0$, our
equation is locally uniformly elliptic in the sense that, almost
surely, $\inf_{V} \lambda(\cdot,\omega)>0$ for each $V\in\mathcal{L}$.

Using ergodicity, we may improve the second part of (F3) to
\[
\sup_{y\in\mathbb R^d} \esssup_{\omega\in\Omega} \bigl\llvert F(0,y,\omega)
\bigr\rrvert < +\infty.
\]
Using then the continuity of $F$ and intersecting the event on which
the latter holds over all the rational points of $\mathbb R^d$, we obtain
\[
\esssup_{\omega\in\Omega} \sup_{y\in\mathbb R^d} \bigl\llvert F(0,y,\omega)
\bigr\rrvert < +\infty.
\]
Applying also (F2), this yields, for $C_0:= \esssup_{\omega\in\Omega
} \llvert F(0,y,\omega) \rrvert $ and all $M\in{\mathbb{S}^d}$,
%e2.3 #&#
\begin{equation}
\label{sup} \esssup_{\omega\in\Omega} \sup_{y\in\mathbb R^d} \bigl\llvert
F(M,y,\omega) \bigr\rrvert \leq C_0 + \Lambda\tr(M_++M_-) \leq
C\bigl(1+|M|\bigr).
\end{equation}

The second statement in assumption (F3) is taken in order that the
comparison principle hold in each bounded domain for the operator
$F(\cdot,\cdot,\omega)$ and for every $\omega\in\Omega$. This is
a consequence of the local uniform ellipticity of $F$ and standard
comparison results (see~\cite{CIL}).

\subsection*{A brief remark concerning viscosity solutions}
All differential inequalities in this paper are to be interpreted in
the viscosity sense (cf.~\cite{CC,CIL}). We remark that, while it is
not obvious---in fact, it is \textit{equivalent} to the comparison
principle---we have transitivity of inequalities in the viscosity sense
(see~\cite{A}, Lemma~3.2). For example, if $V\in\mathcal{L}$ and
$u,-v\in\USC(V)$ satisfy
\[
F \bigl(D^2u,y,\omega \bigr) \geq0 \quad\mbox{and}\quad F
\bigl(D^2u,y,\omega \bigr) \leq 0\qquad\mbox{in } V
\]
then formally it follows that for $w:=u-v$ we have
%e2.4 #&#
\begin{equation}
\label{wsat} 0 \leq F \bigl(D^2u,y,\omega \bigr) - F
\bigl(D^2u,y,\omega \bigr) \leq{\mathcal {P}}^+_{\lambda(x,\omega),\Lambda}
\bigl(D^2w \bigr).
\end{equation}
We emphasize that we may also deduce ${\mathcal{P}}^+_{\lambda
(x,\omega),\Lambda}(D^2w) \geq0$ in the viscosity sense, and make
other similar formal deductions rigorous, using~\cite{A}, Lemma~3.2.

\subsection*{Pointwise notions of twice differentiability and
$C^{1,1}$} We require the following pointwise regularity notions. We
say that $u \in C(B(0,1))$ is \textit{twice differentiable} at $x\in
B_1$ if there exist $(X,p) \in{\mathbb{S}^d}\times\mathbb R^d$ such that
\[
\limsup_{r\to0} \sup_{y\in B_r(x)} r^{-2}
\biggl\llvert u(y) - u(x) - p \cdot(y - x) - \frac{1}{2} (y-x)\cdot X (y-x)
\biggr\rrvert = 0,
\]
in which case we write $D^2 u(x):= X$ and $Du(x):=p$. We also say that
$u$ is $C^{1,1}$ \textit{on a set} $E \subseteq B(0,1)$ if $u$ is
differentiable at each point of $E$ and
\[
\sup_{x\in E} \sup_{y\in B_1} \frac{\llvert  u(y) - u(x) - Du(x)
\cdot(y - x) \rrvert }{|x-y|^2} < +
\infty.
\]
A function $u$ is \textit{semiconcave} if there exist $k>0$ such that
the map $x \mapsto u(x) - k |x|^2$ is concave.
In this paper, we rely many times on the observation that, for any
$a>0$, a semiconcave function is $C^{1,1}$ on the set of points at
which it can be touched from below by a $C^2$ functions with Hessian
bounded by $a$. Moreover, by Rademacher's theorem and the Lebesgue
differentiation theorem, any $C^{1,1}$ function on a set $E$ is twice
differentiable at (Lebesgue) almost every point of $E$.

\subsection*{Infimal convolution}
We recall a standard tool (cf.~\cite{CIL,CC} for details) in the
theory of viscosity solutions. We denote the infimal convolution of~$u
\in\LSC(B_1)$ by
%e2.5 #&#
\begin{equation}
\label{einfconv} u_\varepsilon(x):= \inf_{y \in B_1} \biggl( u(y)
+ \frac
{2}{\varepsilon}|x - y|^2 \biggr).
\end{equation}
The function $u_\varepsilon$ is more regular than $u$ and, in
particular, is semiconcave. It is a good approximation to $u$ in the
sense that $u_\varepsilon\to u$ locally uniformly in $B_1$ as
$\varepsilon\to0$. Moreover, if $f,\lambda\in C(B_1)$, $\lambda> 0$ and
\[
{\mathcal{P}}^+_{\lambda(x), \Lambda} \bigl(D^2 u \bigr) \geq f
\qquad\mbox{in } B_1,
\]
then there exist sequences of functions $\lambda'_\varepsilon,
f'_\varepsilon\in C(B_1)$ which converge locally uniformly to $\lambda
$ and $f$, respectively, as $\varepsilon\to0$, such that
$u_\varepsilon$ satisfies
\[
{\mathcal{P}}^+_{\lambda'_\varepsilon(x), \Lambda} \bigl(D^2 u_\varepsilon \bigr)
\geq f'_\varepsilon\qquad\mbox{in } B_{1-r_\varepsilon},
\]
where $r_\varepsilon\to0$ as $\varepsilon\to0$. We refer to~\cite
{CC,CIL} for details. For us, the principle utility of these
approximations is the semiconcavity of $u_\varepsilon$. If
$u_\varepsilon$ can be touched from below by a smooth function
$\varphi$ at some point $z \in B_1$, then $u_\varepsilon$ is
$C^{1,1}$ at $z$, with norm depending only $\varepsilon$ and $|D^2
\varphi(z)|$. See~\cite{CC}, Theorem~5.1.

\subsection*{Statements of the ergodic theorems}
We next recall the two versions of the (multiparameter) ergodic theorem
used in this paper. The first is nearly a consequence of the second,
but since it is simpler we give it separately. A nice proof can be
found in Becker~\cite{Be}.

We emphasize that the assumptions on the probability space $(\Omega,\mathcal{F},\mathbb{P})$, in particular~(\ref{Ppres}) and~(\ref
{erg}), are in force. Recall that $\mathcal{L}$ denotes the set of all
bounded Lipschitz subsets of $\mathbb R^d$.

%
%pr2.1 #&#
\begin{prop}[(Wiener's ergodic theorem)]\label{ET}
Let $f \in L^1(\Omega)$. Then there exists a subset $\Omega_0\in
\mathcal{F}$ of full probability such that, for every $\omega\in
\Omega_0$ and $V\in\mathcal{L}$,
%e2.6 #&#
\begin{equation}
\label{ETeq} \lim_{t\to\infty} \fint_{tV} f(
\tau_y\omega) \,dy = \mathbb {E} [ f ].
\end{equation}
In particular, the map $y\mapsto f(\tau_y\omega)$ belongs to $
L^1_{\mathrm{loc}}(\mathbb R^d)$ for every $\omega\in\Omega_0$.
\end{prop}

The version of Proposition~\ref{ET} proved in~\cite{Be} actually
requires $V$ to be star-shaped with respect to the origin. As is well
known, this restriction may be removed as follows. First we notice that
the conclusion holds for any cube $V=Q$ with sides parallel to the
coordinate axes, since any such cube either contains the origin or has
the property that, for some larger cube $\widetilde Q$, both
$\widetilde Q \setminus Q$ and $\widetilde Q$ are star-shaped with
respect to the origin. Since it holds for such cubes, it holds for an
arbitrary finite disjoint union of them, and hence any $V\in\mathcal
{L}$ by approximation.

We next state the multiparameter subadditive ergodic theorem of Akcoglu
and Krengel~\cite{AK} as modified by Dal Maso and Modica~\cite{DM2},
which requires some further notation. We denote by $\mathcal{U}_0$ the
family of bounded subsets of $\mathbb R^d$. A function $f\dvtx \mathcal
{U}_0\to\mathbb{R}$ is \textit{subadditive} if
\[
f(A) \leq\sum_{j=1}^k
f(A_j),
\]
whenever\vspace*{-1pt} $k\in\mathbb{N}$ and $A,A_1,\ldots,A_k \in\mathcal{U}_0$
are such that $\bigcup_{j=1}^k A_j\subseteq A$, the sets $A_1,\ldots,A_k$
are pairwise disjoint and $\llvert  A \setminus\bigcup_{j=1}^k A_j
\rrvert  = 0$. Let $\mathcal{M}$ be the collection of subadditive
functions $f\dvtx \mathcal{U}_0\to\mathbb{R}$ which satisfy
\[
0 \leq f(A) \leq|A| \qquad\mbox{for every }A \in\mathcal{U}_0.
\]
A \textit{subadditive process} is a function $f\dvtx \Omega\to\mathcal
{M}$. It is sometimes convenient to write $f(A,\omega) = f(\omega
)(A)$, in which case we have $f(A,\tau_y \omega) = f(y+ A,\omega)$.

%
%pr2.2 #&#
\begin{prop}[(Subadditive ergodic theorem)]
\label{SET}
Let $f\dvtx \Omega\to\mathcal{M}$ be a subadditive process. Then there
exists an event $\Omega_0\in\mathcal{F}$ of full probability and a
constant $0 \leq a \leq1$ such that, for every $\omega\in\Omega_0$
and $V\in\mathcal{L}$,
%e2.7 #&#
\begin{equation}
\lim_{t\to\infty} \frac{f(tV,\omega)}{|tV|} = a.
\end{equation}
\end{prop}

This version of the subadditive ergodic theorem is~\cite{AK}, Proposition~1, in the special case that $\mathcal{L}$ is replaced by the
family of all cubes, and we recover the general case by an easy
approximation argument.

%s3 #&#
\section{Regularity in the macroscopic limit}\label{EO}
The classical regularity theory for uniformly elliptic equations (as
developed, e.g., in~\cite{CC}) does not directly help us to
homogenize~(\ref{pde}) because, as $\varepsilon$ becomes small, the
ergodic theorem guarantees that the set where~(\ref{pde}) has very
high ellipticity becomes dense. What we need are estimates which do not
degenerate as $\varepsilon\to0$, and for this it is necessary to
revisit the regularity from the beginning.

What the ergodic theorem ensures is that, almost surely in~$\omega$,
for every $\mu> 0$,
%e3.1 #&#
\begin{equation}
\label{erghelp} \lim_{\varepsilon\to0} \fint_{V \cap\{ \lambda< \mu\} } \lambda
^{-d} \biggl( \frac{x}\varepsilon, \omega \biggr) \,dx =
\mathbb{E} \bigl[ \lambda^{-d} \mathbh{1}_{ \{ \lambda< \mu\}} \bigr].
\end{equation}
In this section we develop a deterministic regularity theory for
solutions of~(\ref{pde}) which will be robust in the almost sure
macroscopic limit $\varepsilon\to0$ by virtue of~(\ref{erghelp}).

Since\vspace*{1pt} the random environment plays no role here, we drop dependence
on~$\omega$. Throughout this section, we consider a continuous function
$\lambda\dvtx \mathbb R^d\to(0,\Lambda]$, and we study the regularity of
subsolutions and/or supersolutions of the extremal operators ${\mathcal
{P}}^\pm_{\lambda(x),\Lambda}$ in bounded\vspace*{1pt} Lipschitz domains $V\in
\mathcal{L}$. Our estimates must depend only on $d$, $\Lambda$ and,
for $\mu> 0$, the quantities
\[
\int_{V \cap\{ \lambda< \mu\} } \lambda^{-d} ( x ) \,dx.
\]
As it is purely deterministic, the regularity developed here is of
independent interest.

The primary goal is to obtain an improvement of oscillation result on
unit scales, giving us a modulus of continuity [and a H\"older estimate
in the macroscopic limit for solutions of~(\ref{pde})]. Our arguments
are loosely based on the arguments in the classical regularity
theory~\cite{CC}, with some nice modifications due to Savin~\cite{S},
but require several new ideas to overcome the degeneracy of the
equation. For instance, ``two important tools'' are introduced in~\cite{CC},
Section~4.1, which are used repeatedly in what has become the
standard proof of the Harnack inequality. In our situation, neither of
these tools can be applied in a straightforward way.

First, showing that an appropriate barrier (or ``bump'') function
exists---which is easy in the uniformly elliptic situation (see~\cite{CC}, Lemma~4.1)---is a very nontrivial matter. We construct a barrier
by touching a candidate function from below by translations of the
singular function~$|x|^{-\alpha}$ with~$\alpha\gg1$ and then
adapting the proof of the ABP inequality to show that the corresponding
contact set would be too large if the function failed to be a barrier.
Second, the measure-theoretic argument involving the Calder\'
on--Zygmund cube decomposition must be altered due to the presence of
``bad'' cubes of high ellipticity, and we use an alternative idea based
on the Besicovitch covering theorem.

The development is essentially self-contained and depends also on some
novel uses of the area formula for Lipschitz functions, similar to the
proof of the ABP inequality, and partially inspired by~Savin~\cite
{S,S2}. The main result of this section is the following proposition.

%
%pr3.1 #&#
\begin{prop}[(Decay of oscillation)]
\label{poscillation}
There exists $\delta> 0$, depending only on $d$ and $\Lambda$, such
that if $0< \mu< \frac12$ and
\[
\int_{B_1 \cap\{ \lambda< \mu\}} \lambda^{-d}(x) \,dx < \delta,
\]
then there exist constants $0 < \tau< 1$ depending only on $d$,
$\Lambda$ and $\mu$, such that for all $\alpha> 0$ and $u \in
C(B_1)$ satisfying
\[
{\mathcal{P}}^+_{\lambda(x), \Lambda} \bigl(D^2 u \bigr) \geq-\alpha\quad
\mbox{and}\quad{\mathcal{P}}^-_{\lambda(x), \Lambda} \bigl(D^2 u \bigr)
\leq \alpha\qquad\mbox{in } B_1,
\]
we have
\[
\mathop{\operatorname{osc}}_{B_{1/8}} u \leq\tau\mathop{\operatorname{osc}}_{B_1}
u + \alpha.
\]
\end{prop}

Proceeding with the proof of Proposition~\ref{poscillation}, we begin
with three applications of area formula for Lipschitz functions
(cf.~\cite{EG}), which asserts that
\[
\bigl|f(E)\bigr| = \int_E \bigl\llvert \det Df(x) \bigr\rrvert \,dx
\]
for all Lebesgue measurable sets $E \subseteq\mathbb{R}^d$ and
injective Lipschitz maps $f \dvtx  E \to\mathbb{R}^d$\hspace*{-0.8pt}.

The first is the Alexandroff--Bakelman--Pucci (ABP) inequality. The
version we give here is not new: it is actually a corollary to the
proof of \cite{CC}, Theorem~3.2. We include a proof below both for
completeness and in order to introduce the style of argument we use
below, in a more complicated form, to obtain the barrier function. The
argument here is much simpler than the one in~\cite{CC}, which is due
to the observation that a semiconcave function is necessarily $C^{1,1}$
on the set where it can be touched from below by a plane.

%
%pr3.2 #&#
\begin{prop}[(ABP inequality)]
\label{pabp}
Let $f \in C(B_1)$ and suppose that $u\in\LSC(\overline B_1)$ satisfies
\[
\cases{ {\mathcal{P}}^+_{\lambda(x), \Lambda} \bigl(D^2 u
\bigr) \geq- f, &\quad in $B_1$,
\vspace*{3pt}\cr
u \geq0, &\quad on $\partial B_1$.}
\]
Then
\[
u_-(0) \leq \biggl( \frac{1}{|B_1|} \int_{\{\Gamma_u = u \}}
\lambda^{-d}(x) f_+^d(x) \,dx \biggr)^{1/d},
\]
where
\[
\Gamma_u(x):= \sup_{p \in\mathbb{R}^d} \inf
_{y \in B_1} \bigl( p \cdot(x - y) - u_-(y) \bigr)
\]
is convex envelope of $-u_-:= \min\{ 0, u \}$.
\end{prop}

\begin{pf}
By approximating $u$ by its infimal convolution, we may assume that $u$
is a semiconcave [it is straightforward to see that the limsup of the
contact sets for $u_\varepsilon$ in~(\ref{einfconv}) are contained
in the contact set for $u$].

Let $a:= - u(0)$ and assume $a>0$. Since $u \geq0$ on $\partial B_1$,
for every $p\in B_a$, there exists $\bar z(p)\in B_1$ such that
$u(\bar z(p)) < 0$ and the map $x \mapsto-u_-(x) - p \cdot x$
attains its infimum over $B_1$ at $\bar z(p)$. Note that we can
arrange for $\bar z\dvtx  B_a\to B_1$ to be Lebesgue measurable by
choosing $\bar z$, say, lexicographically among the minimizers
closest to the origin. Since $u$ is semiconcave and can be touched from
below by a plane on $A:= \bar z( B_a)$, it is $C^{1,1}$ on $A$.
In particular, $\bar z$ has a Lipschitz inverse $\bar p\dvtx  A
\to B_1$ given by $\bar p(z):= Du(z)$.

By Rademacher's theorem (cf.~\cite{EG}) and the Lebesgue
differentiation theorem,~$u$ is twice differentiable at Lebesgue almost
every point of~$A$. At every such~$z\in A$, we have that $D^2u(z) \geq
0$, since $u$ can be touched from below by a plane at~$z$, and so the
supersolution inequality gives
\[
-f(z) \leq{\mathcal{P}}^+_{\lambda(z),\Lambda} \bigl(D^2u(z) \bigr) = -
\lambda(z) \tr \bigl( D^2u(z) \bigr).
\]
Thus, at almost every point $z \in A$,
%e3.2 #&#
\begin{equation}
0 \leq D^2u(z) \leq\lambda^{-1}(z) f_+(z) I.
\end{equation}
The area formula yields
\[
a^d|B_1| = |B_a| = \int
_A \bigl|\det D \bar p(x)\bigr| \,dx = \int_A
\bigl|\det D^2u(x)\bigr| \,dx \leq\int_A
\lambda^{-d}(x) f_+^d(x) \,dx.
\]
Since $A \subseteq\{ \Gamma_u = u \}$, we obtain the proposition.
\end{pf}

Using a more sophisticated version of the above argument, we next
construct the barrier function, which below plays a critical role in
the proof of Lemma~\ref{liter} below, similar to that of the ``bump''
function in~\cite{CC}. It is also needed in the next section in proof
of Theorem~\ref{H} to verify that the limit function satisfies the
Dirichlet boundary condition.

%
%le3.3 #&#
\begin{lem}
\label{lbarrier}
For each $0 < r <\frac{1}2$, there exists $\delta> 0$, depending only
on $d$~and~$\Lambda$, such that if $0 < \mu< \frac12$ and
\[
\int_{B_1 \cap\{ \lambda< \mu\}} \lambda^{-d}(x) \,dx < \delta
r^d,
\]
then there exists a constant $\beta> 0$, depending only on $d$,
$\Lambda$, $r$ and $\mu$, such that for each $u\in\LSC(\overline
B_1\setminus B_r)$ satisfying
\[
\cases{ {\mathcal{P}}^+_{\lambda(x), \Lambda} \bigl(D^2 u \bigr) \geq- 1, &
\quad in $B_1 \setminus\overline B_r$,
\vspace*{3pt}\cr
u \geq0, &\quad on $\partial B_1$,
\vspace*{3pt}\cr
u \geq\beta, &\quad on $\partial B_r$,}
\]
we have $u > 0$ on $B_{1 - r} \setminus B_r$.
\end{lem}

\begin{pf}
By approximating $u$ by its infimal convolution, we may assume that $u$
is semiconcave. Define
\[
\alpha:= \frac{2 (d-1) \Lambda+2}{\mu} \quad\mbox{and}\quad \beta:= \biggl(
\frac{4}{r} \biggr)^{\alpha}.
\]

The idea is to show that if $u$ is negative somewhere in
$B_{1-r}\setminus B_r$, then it can be touched from below by (too many)
small translations of the singular function $\phi(x):= 2^{\alpha}
|x|^{-\alpha}$. Let us suppose that $u(x_0) < 0$ for some $x_0 \in
B_{1-r} \setminus B_r$. As a consequence we find that, for every $y\in
B_{r/2}$, the map $x\mapsto u(x) - \phi(x-y)$ attains its infimum at
some point $\bar z(y) \in B_1\setminus B_r$. To verify this we
check that, for $y\in B_{r/2}$,
%e3.3 #&#
\begin{eqnarray}
\inf_{x\in\partial B_{1}} \bigl(u(x) - \phi(x-y) \bigr) &\geq& -
2^\alpha \biggl\llvert 1-\frac{r}2 \biggr\rrvert ^{-\alpha}
\nonumber\\[-8pt]\\[-8pt]
&\geq&- \phi(x_0-y) > u(x_0) - \phi(x_0-y)\nonumber
\end{eqnarray}
and, by our choice of $\beta$,
%e3.4 #&#
\begin{equation}
u(x) \geq\beta\geq\phi(x-y) \qquad\mbox{for every } x\in \partial
B_r.
\end{equation}
It is easy to arrange for the function $\bar z\dvtx B_{r/2} \to
B_1\setminus B_r$ to be Lebesgue measurable. To obtain the
contradiction, we eventually apply the area formula to the inverse of~$\bar z$. Most of the rest of the argument is concerned with
showing that the image of $\bar z$ is contained in the region
where $\lambda$ is small, that $\bar z$ has an inverse $\bar y$ and estimating the determinant of the Jacobian of $\bar y$.

The Hessian of $\phi$ is given by
%e3.5 #&#
\begin{equation}
\label{eHessphi} D^2\phi(x) = \alpha2^\alpha|x|^{-\alpha-2}
\biggl( (\alpha+1) \frac{x\otimes x}{|x|^2} - \biggl( I - \frac{x\otimes x}{|x|^2} \biggr)
\biggr)
\end{equation}
and thus
\[
\mbox{the eigenvalues of } D^2\phi(x) = \alpha2^\alpha|x|^{-\alpha
-2}
\cdot %
\cases{(\alpha+1), &\quad with multiplicity $1$,
\vspace*{3pt}\cr
-1, &\quad with multiplicity $d-1$.}
\]
The differential inequality for $u$ at $z=\bar z(y)$ yields
\begin{eqnarray*}
- 1 \leq{\mathcal{P}}^+_{\lambda(z),\Lambda} \bigl(D^2 \phi(z - y) \bigr)
& =& \alpha2^\alpha|z-y|^{-\alpha-2} \bigl( (d-1)\Lambda-(\alpha+1)
\lambda(z) \bigr)
\\
& \leq&\alpha(\alpha+ 1) 2^\alpha|z - y|^{-(\alpha+ 2)} \biggl(
\frac{\mu}{2} - \lambda(z) \biggr).
\end{eqnarray*}
Using that $2^\alpha|z-y|^{-(\alpha+2)} \geq\frac14$ and $\alpha
(\alpha+1) \geq8/\mu$ and rearranging this, we get
%e3.6 #&#
\begin{equation}
\lambda(z) < \mu.
\end{equation}
We conclude that
%e3.7 #&#
\begin{equation}
\label{Acapt} A:= \bar z(B_{r/2}) \subseteq \bigl\{x\in
B_1 \dvtx  \lambda(x) < \mu \bigr\}.
\end{equation}

Since $u$ is semiconcave and $|D^2 \varphi|$ is bounded in $\mathbb
R^d\setminus B_{r/2}$, we see that $u$ is $C^{1,1}$ on $A$. In
particular, $u$ is differentiable at each point of $A$ and, by
Rademacher's theorem and the Lebesgue differentiation theorem, twice
differentiable at Lebesgue almost every point of $A$. For each $y\in B_{r/2}$,
\[
Du \bigl(\bar z(y) \bigr) = D\phi \bigl(\bar z(y) - y \bigr) = -
\alpha2^\alpha \bigl|\bar z(y) - y\bigr|^{-(\alpha+ 2)} \bigl(\bar z(y) -
y \bigr).
\]
Hence,
%e3.8 #&#
\begin{equation}
\label{enDu} \bigl\llvert Du \bigl(\bar z(y) \bigr) \bigr\rrvert =
\alpha2^\alpha \bigl\llvert \bar z(y) - y \bigr\rrvert
^{-(\alpha+ 1)}
\end{equation}
and substituting this into the previous line yields
\[
Du \bigl(\bar z(y) \bigr) =- \bigl(\alpha2^\alpha
\bigr)^{-1/(\alpha+1)} \bigl\llvert Du \bigl(\bar z(y) \bigr) \bigr\rrvert
^{(\alpha+ 2)/(\alpha+1)} \bigl(\bar z(y) - y \bigr).
\]
Solving this for $y$, we find that $\bar z$ has Lipschitz inverse
$\bar y\dvtx A \to B_{r/2}$ given by
\[
\bar y(z):= z + \bigl(\alpha2^\alpha \bigr)^{1/(\alpha+1)} \bigl
\llvert Du(z) \bigr\rrvert ^{-(\alpha+ 2)/(\alpha+ 1)} Du(z).
\]
Since $Du(\bar z(y)) = D\phi(\bar z(y) - y) \neq0$ at each
$y\in B_{r/2}$ on $A$, it is clear that $Du \neq0$ on $A$ and thus
$\bar y$ is differentiable at each $z \in A$ at which $u$ is twice
differentiable; at such $z\in A$, we compute
\begin{eqnarray*}
D\bar y(z) &=& I + \bigl( \alpha2^\alpha \bigr)^{1/(\alpha+1)} \bigl
\llvert Du(z) \bigr\rrvert ^{-(\alpha+2)/(\alpha+1)}
\nonumber\\[-8pt]\\[-8pt]
&&\times{} \biggl( D^2u(z) \biggl( I - \frac{\alpha+2}{\alpha+1} \frac
{Du(z)}{|Du(z)|} \otimes\frac{Du(z)}{|Du(z)|} \biggr) \biggr).
\end{eqnarray*}
Using~(\ref{enDu}) we conclude that, at almost every $z\in A$,
%e3.9 #&#
\begin{equation}
\label{eDbary} \bigl\llvert D\bar y(z) \bigr\rrvert \leq C \bigl( 1 +
\bigl( \alpha 2^\alpha \bigr)^{-1} \bigl\llvert \bar y(z) -
z \bigr\rrvert ^{\alpha+2} \bigl\llvert D^2u(z) \bigr\rrvert
\bigr),
\end{equation}
where $C>0$ is a constant depending only on $d$.

It remains to estimate $|D^2u|$ on $A$. Using~(\ref{eHessphi}) and
that $\phi$ touches $u$ from below on $A$, we have, at each $z\in A$
at which $u$ is twice differentiable,
%e3.10 #&#
\begin{equation}
\label{hessdn} D^2u(z) \geq D^2\phi \bigl(z-\bar y(z) \bigr) \geq- \alpha2^\alpha\bigl|z - \bar y(z)\bigr|^{-(\alpha+ 2)} I.
\end{equation}
On the other hand, the differential inequality gives
\[
-1 \leq{\mathcal{P}}^+_{\lambda(z),\Lambda} \bigl( D^2u(z) \bigr) =
\Lambda\tr \bigl( D^2u(z) \bigr)_- - \lambda(z) \tr
\bigl(D^2u(z) \bigr)_+.
\]
A rearrangement of the later yields, in light of~(\ref{hessdn}),
\[
\tr \bigl(D^2u(z) \bigr)_+ \leq \bigl( d\Lambda\alpha2^\alpha\bigl|z-
\bar y(z)\bigr|^{-(\alpha+2)} + 1 \bigr) \lambda^{-1}(z)
\]
and from this we deduce that
%e3.11 #&#
\begin{equation}
\label{hessup} D^2 u(z) \leq C \lambda(z)^{-1} \bigl(
\alpha2^\alpha\bigl|z -\bar y(z)\bigr|^{-(\alpha+ 2)} + 1 \bigr) I,
\end{equation}
where here and in the rest of the proof $C> 0$ depends only on $d$ and
$\Lambda$. Combining~(\ref{hessdn}) and~(\ref{hessup}), we obtain, at
each point $z\in A$ at which $u$ is twice differentiable,
%e3.12 #&#
\begin{equation}
\label{hessest} \bigl\llvert D^2u(z) \bigr\rrvert \leq C
\lambda^{-1}(z) \bigl( \alpha2^\alpha \bigl|z-\bar y(z)\bigr|^{-(\alpha+2)} + 1 \bigr).
\end{equation}
Inserting this into~(\ref{eDbary}) and using that $\alpha\geq1$,
$|z - \bar y(z)| \leq2$ and $\lambda^{-1}(z) \geq\mu^{-1} \geq
2$, we at last deduce
%e3.13 #&#
\begin{equation}
\label{eDbary2} \bigl\llvert D\bar y(z) \bigr\rrvert \leq C \bigl( 1 +
\bigl( \alpha 2^\alpha \bigr)^{-1} \bigl\llvert \bar y(z) -
z \bigr\rrvert ^{\alpha+2} \bigr) \lambda^{-1}(z) + C \leq C
\lambda^{-1}(z)
\end{equation}
at Lebesgue almost every point $z\in A$.

We finally apply the area formula, using~(\ref{Acapt}),~(\ref
{eDbary2}) and the hypothesis of the lemma to conclude that
\begin{eqnarray*}
\bigl( 2^{-d} |B_1| \bigr) r^d&=&
|B_{r/2}| = \int_A \bigl|\det D\bar y(x)\bigr| \,dx
\\
&\leq& C \int_A \lambda^{-d}(x) \,dx \leq C \int
_{B_1 \cap
\{ \lambda< \mu\}} \lambda^{-d}(x) \,dx \leq C\delta
r^d.
\end{eqnarray*}
We get a contradiction if $\delta> 0$ is sufficiently small, depending
on $d$ and $\Lambda$.
\end{pf}

%
%re3.4 #&#
\begin{remark}
\label{rbarrier}
By an easy modification of the above proof, we also obtain an
additional estimate for the barrier. Given $h > 0$, we can modify the
choice of $\delta, \beta> 0$ and also select an $r' \in(r, 1)$
depending only on $d$, $\Lambda$, $\varepsilon$, $\mu$, $r$ and $h$
(but not $\beta$) such that the supersolution $u$ satisfies $u \geq
\beta- h$ on $B_{r'} \setminus B_r$. Later we use this observation to
verify the Dirichlet boundary condition for the limit function in the
proof of homogenization.
\end{remark}

The next lemma, which is inspired by~\cite{S}, Lemma~2.1, follows from
another application of the area formula and a (much easier) variation
of the above argument. It asserts that, if a supersolution can be
touched from below by sufficiently many translations of a fixed
parabola, then the $L^d$ norm of $\lambda^{-1}$ on the set of points
at which the touching occurs cannot be too small.

%
%le3.5 #&#
\begin{lem}
\label{lcontact}
Let $u \in\LSC(B_1)$ satisfy
\[
{\mathcal{P}}^+_{\lambda(x),\Lambda} \bigl(D^2 u \bigr) \geq-1 \qquad\mbox{in } B_1.
\]
Suppose that $a\geq1$ and $V \subseteq\mathbb{R}^d$ such that, for
each $y\in V$, the infimum over $B_1$ of the map $z \mapsto u(z) +
\frac{a}{2} |z - y|^2$ is attained. Let $W \subseteq B_1$ denote the
union over $y\in V$ of the subset of $B_1$ at which this map attains
its minimum. Then there exists a constant $\delta> 0$, depending only
on $d$ and $\Lambda$, such that
\[
\int_W \lambda^{-d}(x) \,dx \geq\delta|V|.
\]
\end{lem}

\begin{pf}
By replacing $u$ by $u+\alpha|x|^2$ and letting $\alpha\to0$, we may
suppose that, for some small $\eta>0$, $W \subseteq B_{1-\eta}$ and,
for every $y\in B_1$,
%e3.14 #&#
\begin{equation}
\label{strictmin} \min_{ z\in\partial B_{1-\eta}} \biggl(u(z) + \frac{a}2|z-y|^2
\biggr) > \inf_{z\in W} \biggl(u(z) + \frac{a}2|z-y|^2
\biggr).
\end{equation}
As in the proofs of Proposition~\ref{pabp} and Lemma~\ref{lbarrier}, we may assume that $u$ is semiconcave by infimal
convolution approximation. Indeed, due to~(\ref{strictmin}), the set
$V$ is essentially unchanged by the infimal convolution approximation,
while the set $W$ is unchanged or possibly smaller.

Select a Lebesgue-measurable function $\bar z\dvtx V\to B_1$ such that
the map $z \mapsto u(z) + \frac{a}{2} |z - y|^2$ attains its infimum
in $B_1$ at $z=\bar z(y)$. The function $u$ is $C^{1,1}$ on
$A:=\bar z(V)$ and $\bar z$ has a Lipschitz inverse
$\bar y$ given by
\[
\bar y(z)\dvtx  = z + \frac{1}a Du(z).
\]
By Rademacher's theorem and the Lebesgue differentiation theorem, $u$
is twice~differentiable at almost every point of $z\in A$ and, at such
$z$, we have \mbox{$D^2u(z) \geq-aI$},
%e3.15 #&#
\begin{equation}
D\bar y(z) = I + \frac{1}a D^2u(z) \geq0
\end{equation}
as well as
\begin{eqnarray*}
-\lambda(z) \tr \bigl(D\bar y(z) \bigr) &=& {\mathcal {P}}^+_{\lambda(z),\Lambda}
\bigl(D\bar y(z) \bigr) = {\mathcal{P}}^+_{\lambda(z),\Lambda} \biggl(I +
\frac{1}aD^2u(z) \biggr)
\\
&\geq&\frac{1}a {\mathcal{P}}^+_{\lambda
(z),\Lambda} \bigl(D^2u(z)
\bigr) + {\mathcal{P}}^-_{\lambda
(z),\Lambda} (I) \geq-\frac{1}{a} - \Lambda d
\end{eqnarray*}
and, therefore,
%e3.16 #&#
\begin{equation}
0 \leq D\bar y(z) \leq\frac{1}{\lambda(z)} \biggl(\frac{1}a+ \Lambda
d \biggr).
\end{equation}
An application of the area formula for Lipschitz functions gives
\[
|V| = \int_{A} \bigl| \det D \bar y(x)\bigr| \,dx \leq \biggl(
\frac{1}a + \Lambda d \biggr)^{d} \int_{A}
\lambda^{-d}(x) \,dx
\]
from which we obtain the lemma, using that $a\geq1$ and $A \subseteq W$.
\end{pf}

We now give the proof of the decay of oscillation.
\begin{pf*}{{Proof of Proposition~\ref{poscillation}}}
Now suppose that $0 < \mu< \frac12$ and
%e3.17 #&#
\begin{equation}
\label{mu} \int_{B_1 \cap\{ \lambda< \mu\}} \lambda^{-d}(x) \,dx <
\delta:= \frac{1}{6 N_d} |B_{1/6}| \min \bigl\{ 8^{-d}
\delta_1, 32^{-d} \delta_2 \bigr\},
\end{equation}
where $\delta_1$ is from Lemma \ref{lbarrier}, $\delta_2$ is from
Lemma \ref{lcontact} and $N_d$ is the constant from the Besicovitch
covering theorem in dimension $d$.

We first make a reduction by noticing that, to obtain the proposition
for $\tau:= (1-1/k)$, it suffices to consider $v \in C(B_1)$ which satisfies
%e3.18 #&#
\begin{equation}
\label{eoscalpha1} \mathcal P^+_{\lambda(x), \Lambda} \bigl(D^2 v \bigr) \geq-1
\quad\mbox{and}\quad\mathcal P^-_{\lambda(x),\Lambda} \bigl(D^2 v \bigr)
\leq1 \qquad\mbox{in } B_1
\end{equation}
and
%e3.19 #&#
\begin{equation}
\label{eoscineq} \mathop{\operatorname{osc}}_{B_1} v \leq1 +
\mathop{\operatorname{osc}}_{B_{1/8}} v
\end{equation}
and to show that $\operatorname{osc}_{B_1} v < k$. Indeed, suppose $u$
satisfies the hypotheses of the proposition and
\[
\mathop{\operatorname{osc}}_{B_{1/8}} u > \bigl(1 - k^{-1} \bigr)
\mathop{\operatorname{osc}}_{B_1} u + \alpha.
\]
Then $\alpha< k^{-1} \operatorname{osc}_{B_1} u$, and so if we set
$v:=ku/ \operatorname{osc}_{B_1} u$, then we see that $v$ satisfies
(\ref{eoscalpha1}), (\ref{eoscineq}) and $\operatorname{osc}_{B_1}
v = k$.

Define, for every $\kappa>0$,
\[
A_\kappa:= \biggl\{ x\in B_{1} \dvtx  \exists y \in
B_1, v(x) + \frac{\kappa}{2}|x - y|^2 = \inf
_{z\in B_1} \biggl( v(z) + \frac
{\kappa}{2}|z - y|^2
\biggr) \biggr\}.
\]
In other words, $A_\kappa$ is the set of points in $B_{1}$ at which
$v$ can be touched from below by a paraboloid with Hessian $-\kappa I$
and vertex in $B_1$. To prove the desired estimate on~$v$, it is enough
to show that
%e3.20 #&#
\begin{equation}
\label{Akapwts} |A_\kappa\cap B_{1/6}| \geq
\tfrac23|B_{1/6}|
\end{equation}
for some $\kappa> 0$ depending only on~$d$,~$\Lambda$ and~$\mu$.
Indeed, if we could show this, then using that $-v$ satisfies the same
hypotheses as $v$ and applying~(\ref{Akapwts}) to both functions, we
find a point $x\in B_{1/6}$ which can be touched from above and below
by parabolas with opening $\kappa$. That is, we could conclude that
there exist $x\in B_{1/6}$ and $y_1, y_2 \in B_1$ such that, for all
$z\in B_1$,
\[
v(x) + \frac{\kappa}{2}|x - y_1|^2 -
\frac{\kappa}{2}|z - y_1|^2 \leq v(z) \leq v(x) -
\frac{\kappa}{2}|x - y_2|^2 + \frac{\kappa
}{2}|z -
y_2|^2.
\]
This implies that $|v(x) - v(z)| \leq2\kappa$ for all $z\in B_1$, and
thus $\operatorname{osc}_{B_1} v \leq4 \kappa$, which is the desired
estimate.

In order to obtain~(\ref{Akapwts}) for some $\kappa> 0$ depending on
the appropriate quantities, we observe first that~(\ref{eoscineq})
implies that $A_{576} \cap B_{1/6} \neq\varnothing$. Indeed, $(\frac{1}6-\frac{1}8)^2 = 1/576$ and so $v$ can be touched from below in
$B_{1/6}$ by the parabola $-576|x-y|^2$, where $y \in\overline
B_{1/8}$ is such that $v(y)=\min_{\overline B_{1/8}} v$. We then
repeatedly apply Lemma~\ref{liter} below to obtain the desired result
for $\kappa=576 \cdot\theta^n$, where $n:= \lceil|B_1| / \eta
\rceil$ and $\theta,\eta>0$ are given in the statement of the
lemma. The proof of Proposition~\ref{poscillation} is now complete,
pending the verification of Lemma~\ref{liter}.
\end{pf*}

The following lemma contains the measure theoretic information needed
to conclude the proof of Proposition~\ref{poscillation}. In the
classical regularity theory, this step traditionally relies on the
Calder\'on--Zygmund cube decomposition (as in the proof of~(4.12)
in~\cite{CC}). Since we did not immediately see how to adapt it, and
for the sake of variety, we instead use an alternative tool: the
Besicovitch covering theorem.\footnote{Luis Silvestre has since
pointed out to us that the Calder\'on--Zygmund decomposition argument
in~\cite{CC} may indeed be suitably modified to prove Lemma~\ref
{liter} and that the best choice is the Vitali covering theorem, which
can be used in a similar yet simpler way than the Besicovitch covering
theorem.} The argument also relies in a crucial way on Lemmas~\ref
{lbarrier} and~\ref{lcontact}.

%
%le3.6 #&#
\begin{lem} \label{liter}
Let $\mu$, $v$ and $A_\kappa$ be as in the proof of Proposition~\ref
{poscillation}. There exist constant~$\theta>1$ and $\eta> 0$,
depending only on~$d$,~$\Lambda$ and~$\mu$, such that if $\kappa\geq
1$, $A_\kappa\cap B_{1/6} \neq\varnothing$ and $|A_\kappa\cap B_{1/6}
| < \frac{2}{3}|B_{1/6}|$, then $|A_{\theta\kappa}\cap B_{1} | \geq
|A_\kappa\cap B_{1}| + \eta$.
\end{lem}
\begin{pf}
Consider the collection~$\mathcal{B}$ of balls~$B_r(x) \subseteq B_1$
such that $B_{r/2}(x) \subseteq B_1 \setminus A_\kappa$ and $\partial
B_{r/2}(x) \cap A_\kappa\neq\varnothing$. Note that since $A_\kappa
\cap B_{1/6} \neq\varnothing$ and $A_\kappa$ is closed, every point
of~$B_{1/6}\setminus A_\kappa$ is the center of some ball in~$\mathcal
B$. According to the Besicovitch covering theorem, we may select a
countable subcollection $\{ B_{r_k}(x_k) \}_{k\in\mathbb{N}}
\subseteq\mathcal{B}$ that covers $B_{1/6} \setminus A_\kappa$ and
such that each point $x \in B_1$ belongs to at most $N_d$ balls.

We say that the ball $B_{r_k}(x_k)$ is \textit{good} if
\[
\frac{1}{|B_{r_k}(x)|}\int_{B_{r_k}(x_k) \cap\{ \lambda< \mu\}} \lambda^{-d} (x) \,dx<
\min \bigl\{ 8^{-d} \delta_1, 32^{-d}
\delta_2 \bigr\}
\]
and set $G:= \{ k\in\mathbb{N} \dvtx   B_{r_k}(x_k)$ is
good$\}$, where $\delta_1, \delta_2>0$ are as in the proof of
Proposition~\ref{poscillation}. We claim that at least half of the
Lebesgue measure of $B_{1/6}\setminus A_\kappa$ consists of points
which belong to good balls, that is,
%e3.21 #&#
\begin{equation}
\label{goodballz} \biggl\llvert \bigcup_{k\in G}
B_{r_k}(x_k) \biggr\rrvert > \frac{1}2 \llvert
B_{1/6} \setminus A_\kappa\rrvert \geq\frac{1}6|B_{1/6}|.
\end{equation}
Indeed, if~(\ref{goodballz}) were false, then $\sum_{k\notin G}
|B_{r_k}| \geq\frac{1}2\llvert B_{1/6} \setminus A_\kappa\rrvert  \geq
\frac{1}6 |B_{1/6}|$ and so
\begin{eqnarray*}
\int_{B_1 \cap\{ \lambda< \mu\}} \lambda^{-d}(x) \,dx &\geq&
\frac
{1}{N_d} \sum_{k\notin G} \int
_{B_{r_k}(x_k) \cap\{ \lambda< \mu
\}} \lambda^{-d}(x) \,dx
\\
&\geq&\frac{1}{N_d}\min \bigl\{ 8^{-d} \delta_1,
32^{-d} \delta_2 \bigr\} \sum_{k\notin G}
|B_{r_k}|
\\
& \geq&\frac{1}{6N_d} \min |B_{1/6}| \bigl\{
8^{-d} \delta_1, 32^{-d} \delta_2 \bigr
\},
\end{eqnarray*}
which contradicts~(\ref{mu}). Therefore, in light of the Besicovitch
covering, it is enough to show that $|B_{r_k/2}(x_k) \cap A_{\theta
\kappa}| \geq\eta|B_{r_k}(x_k)|$ for each good ball $B_{r_k}(x_k)$
and some constants $\theta, \eta> 0$ depending only on $d$, $\Lambda
$ and $\mu$.

Fix a good ball $B_r(x):= B_{r_k}(x_k)$ and choose $z_1 \in\partial
B_{r/2}(x) \cap A_\kappa$. By the definition of $A_\kappa$, we can
touch $z_1$ by a paraboloid of Hessian $-\kappa I$: there exists
$y_1\in B_1$ such that
%e3.22 #&#
\begin{equation}
\label{vminzz} v(z_1) + \frac{\kappa}{2} |z_1 -
y_1|^2 = \inf_{z \in B_1} \biggl( v(z) +
\frac{\kappa}{2} |z - y_1|^2 \biggr).
\end{equation}
We argue that, by making this paraboloid steeper and wiggling the
vertex, we may touch the function $v$ at a positive proportion of
points inside of $B_r(x)$. A key role is played by Lemma~\ref
{lbarrier}, which keeps the touching points near the center and away
from the boundary of $B_r(x)$ as well as by Lemma~\ref{lcontact},
which ensures that we can touch a positive proportion of points by
wiggling the vertex of the paraboloid.

Using that $B_r(x)$ is good and applying (a properly scaled) Lemma~\ref
{lbarrier}, there exists $\beta> 1$, depending only on~$d$,~$\Lambda
$ and~$\mu$, such that the solution $w$ of the Dirichlet problem
\[
\cases{ {\mathcal{P}}^+_{\lambda(x), \Lambda} \bigl(D^2 w \bigr) = - 1, &
\quad in $B_r(x) \setminus\overline B_{r/8}(x)$,
\vspace*{3pt}\cr
w = 0,
&\quad on $\partial B_r(x)$,
\vspace*{3pt}\cr
w = \beta r^2, &\quad
on $\partial B_{r/8}(x)$,}
\]
satisfies $w > 0$ in $\overline B_{r/2}(x) \setminus B_{r/8}(x)$.
Clearly, $w \leq\beta r^2$ in $B_r\setminus B_{r/8}(x)$ by the maximum
principle.
Observe that the function
\[
\varphi(z):= (d\Lambda\kappa+ 2) w - \frac{\kappa}{2} |z -
y_1|^2,
\]
satisfies
%e3.23 #&#
\begin{equation}
{\mathcal{P}}^+_{\lambda(x), \Lambda} \bigl(D^2 \varphi \bigr) \leq-2
\qquad\mbox{in } B_r(x) \setminus B_{r/8}(x).
\end{equation}

The comparison principle implies that the map $z \mapsto v(z) - \varphi
(z)$ attains its infimum in $B_r(x) \setminus B_{r/8}(x)$ at some point
$z=z_2 \in\partial B_r(x) \cup\partial B_{r/8}(x)$. Notice, however,
that it is impossible that $z_2 \in\partial B_r(x)$, since (\ref
{vminzz}), $w\equiv0$ on $\partial B_r(x)$ and $w(z_1)>0$ imply that
\begin{eqnarray*}
&& v(z_1) - \varphi(z_1)
\\
&&\qquad = v(z_1) +
\frac{\kappa}{2} |z_1 - y_1|^2 - (d\Lambda
\kappa+ 2) w(z_1) < \inf_{z \in B_1} \biggl( v(z) +
\frac
{\kappa}{2} |z - y_1|^2 \biggr)
\\
&&\qquad \leq \inf_{z\in\partial
B_r(x)} \biggl( v(z) + \frac{\kappa}{2} |z -
y_1|^2 \biggr) = \inf_{z\in\partial B_r(x)} \bigl(
v(z) - \varphi(z) \bigr).
\end{eqnarray*}
Hence, $z_2\in\partial B_{r/8}(x)$ and so, in particular, $\varphi
(z_2) = - \frac{\kappa}{2} |z_2 - y_1|^2 + (d\Lambda\kappa+2)\beta
r^2$. Using that $w> 0$ in $B_{r/2}(x) \setminus B_{r/8}(x)$, we obtain that
\begin{eqnarray*}
&& \inf_{z\in B_{r/2}(x)\setminus B_{r/8}(x)} \biggl( v(z) + \frac\kappa 2|z-y_1|^2
\biggr)
\\
&&\qquad \geq\inf_{z\in B_{r/2}(x)\setminus B_{r/8}(x)} \bigl( v(z) - \varphi(z) \bigr)
\\
&&\qquad = v(z_2) - \varphi(z_2) = v(z_2) + \frac
\kappa2|z_2-y_1|^2 - ( d\Lambda+2/\kappa )
\kappa\beta r^2.
\end{eqnarray*}
Using this together with~(\ref{vminzz}), $z_1\in\partial B_{r/2}(x)$
and $\kappa\geq1$, we obtain
%e3.24 #&#
\begin{equation}
\label{vminzz2} \qquad \inf_{z \in B_1} \biggl( v(z) + \frac{\kappa}{2}
|z - y_1|^2 \biggr) \geq v(z_2) + \frac
\kappa2|z_2-y_1|^2 - ( d\Lambda+2 ) \kappa
\beta r^2.
\end{equation}
It follows that, if we set $\gamma:= 16\beta(d\Lambda+2)+1$, then
for every $y_2\in B_{r/8}(x)$, the function
\[
\psi(z):= v(z) + \frac{\kappa}{2} |z - y_1|^2 +
\frac{\gamma\kappa
}{2}|z-y_2|^2
\]
satisfies $\psi(z_2) < \min_{B_1\setminus B_{r/2}(x)} \psi$ and,
therefore, must attain its infimum over $B_1$ somewhere in $B_{r/2}(x)$.

Consider the function $\bar z\dvtx B_{r/8}(x) \to B_1$ given by
$\bar z( y) = (y_1+\gamma y)/(1+\gamma)$ and observe by
completing the square that, for some $a\in\mathbb{R}$,
\[
\frac{\kappa}{2} |z - y_1|^2 + \frac{\gamma\kappa}{2} |z -
y_2|^2 = \frac{(\gamma+1) \kappa}{2} \bigl|z - \bar z(y_2)\bigr|^2 + a \qquad \mbox{for all } z\in\mathbb
R^d.
\]
It follows that the map $z\mapsto v(z) + \frac{1}2(\gamma+1)\kappa|z
- \bar z(y_2)|^2$ attains its infimum in $B_1$ at some point of
$B_{r/2}(x)$. Since $\gamma\geq1$, and thus $\gamma/(\gamma+1) \geq
\frac{1}2$, we deduce that
%e3.25 #&#
\begin{equation}
\label{ewigg} \bigl|\bar z(B_{r/8})\bigr| \geq2^{-d}
\bigl|B_{r/8}(x)\bigr|.
\end{equation}

We have succeeded in touching the function $v$ by steepening the
paraboloid and wiggling the vertex. Now an application of Lemma~\ref
{lcontact}, using~(\ref{ewigg}) and that $B_r(x)$ is a good ball,
ensures that we have actually touched a positive proportion of points
in $B_r(x)$. We obtain
\begin{eqnarray*}
2^{-d} \delta_2 \bigl|B_{r/8}(x)\bigr| & \leq&\int
_{B_{r/2}(x) \cap A_{(\gamma
+1) \kappa}} \lambda^{-d}(x) \,dx
\\
& \leq&32^{-d} \delta_2 \bigl|B_r(x)\bigr| +
\mu^{-d} \bigl|B_{r/2}(x) \cap A_{(\gamma+1) \kappa}\bigr|,
\end{eqnarray*}
which implies $|B_{r/2}(x) \cap A_{(\gamma+1)\kappa}| \geq\mu^d
2^{-d}(1 - 2^{-d}) \delta_2 |B_{r/8}(x)|$, as desired.
\end{pf}

%s4 #&#
\section{Homogenization} \label{HO}
The proof of homogenization follows the approach of~\cite{CSW},
although we have reorganized the argument for clarity and simplicity as
well as to accommodate the modifications required to handle the
nonuniform elliptic case. The strategy relies on an application of the
subadditive ergodic theorem to a certain quantity involving the
obstacle problem. The proof has three steps:
\begin{enumerate}[(3)]
\item[(1)] Identifying $\overline F$: by applying the subadditive
ergodic theorem to the Lebesgue measure of the contact set of a certain
obstacle problem, we build the effective operator $\overline F$.

\item[(2)] Building approximate correctors: with the help of the
effective regularity results, we compare the solutions of the obstacle
problem to the solution of the Dirichlet problem with zero boundary
conditions and show that the latter act as approximate correctors.

\item[(3)] Proving convergence: using the approximate correctors, the
classical perturbed test function method allows us to conclude.
\end{enumerate}

\subsection*{Step one: Identifying $\overline F$ via the obstacle problem}
Following~\cite{CSW}, we introduce, for each bounded Lipschitz
domain~$V\in\mathcal{L}$, the obstacle problem (with the zero
function as the obstacle):
%e4.1 #&#
\begin{equation}
\label{obst} \cases{\min \bigl\{ F \bigl(D^2w,y,\omega \bigr),  w
\bigr\} = 0, &\quad in $V$,
\vspace*{3pt}\cr
w = 0, &\quad on $\partial V$.}
\end{equation}
Some important properties of~(\ref{obst}) are reviewed in
\hyperref[OO]{Appendix}. It is well known that~(\ref{obst}) has a unique
viscosity solution, which we denote by $w=w(y,\omega;  V,F)$. We
often write $w=w(y,\omega; V)$ or simply $w=w(y,\omega)$ if we do not
wish to display the dependence on $F$ or $V$.

The set $\mathcal{C}(V,\omega):= \{ y\in V   \dvtx    w(y,\omega; V) =
0 \}$ of points where $w$ touches the obstacle is called the \textit{contact set}. We write $\mathcal{C}(V,\omega; F)$ if we wish to
display the dependence on $F$. The Lebesgue measure of this set is an
important quantity, and we denote it by
%e4.2 #&#
\begin{equation}
\label{cntctset} m(V,\omega): = \bigl\llvert \mathcal{C}(V,\omega) \bigr\rrvert.
\end{equation}
We check that $m$ satisfies the hypotheses of the subadditive ergodic
theorem (Proposition~\ref{SET}). First we observe from the
monotonicity of the obstacle problem [see~(\ref{obst-mono})], that for
all $V,W\in\mathcal{L}$ and $\omega\in\Omega$,
%e4.3 #&#
\begin{equation}
\label{Cs-mono} V \subseteq W \qquad\mbox{implies that }\mathcal{C}(W,
\omega ) \cap V \subseteq\mathcal{C}(V,\omega).
\end{equation}
Immediate from~(\ref{Cs-mono}) is the subadditivity of~$m$. That is,
for all $V, V_1,\ldots, V_{k} \in\mathcal{L}$ such that $\bigcup_{j=1}^k V_j \subseteq V$, the sets $V_1,\ldots,V_k$ are pairwise
disjoint and $| V \setminus\bigcup_{j=1}^k V_j | = 0$, we have
%e4.4 #&#
\begin{equation}
\label{msa} m(V,\omega) \leq\sum_{j=1}^k
m(V_j,\omega).
\end{equation}
According to (F1), $m$ is stationary, that is,
\[
m(V,\tau_y\omega) = m(y+V,\omega)
\]
for every $y\in\mathbb R^d$ and $V\in\mathcal{L}$. We may easily
extend $m$ to $\mathcal{U}_0$ by defining, for every $A \in\mathcal{U}_0$,
\[
\tilde m(A,\omega):= \inf \bigl\{ m(V,\omega) \dvtx  V \in \mathcal{L}\mbox{ and }
A \subseteq V \bigr\}.
\]
This extension agrees with $m$ on $\mathcal{L}$ by~(\ref{Cs-mono})
and it is easy to show that the subadditivity and stationarity
properties are preserved.

We now obtain the following lemma.

%
%le4.1 #&#
\begin{lem} \label{obstH}
There exists an event $\Omega_2 \in\mathcal{F}$ of full
probability and a deterministic constant $\bar{m}\in\mathbb{R}$
such that, for every $\omega\in\Omega_2$ and Lipschitz domain $V
\subseteq\mathbb R^d$,
%e4.5 #&#
\begin{equation}
\label{obstHe} \lim_{t\to\infty} \frac{1}{t^d} m(tV,\omega) =
\bar m|V|.
\end{equation}
\end{lem}
\begin{pf}
In light of the remarks preceding the statement, the lemma follows from
the subadditive ergodic theorem (Proposition~\ref{SET}).
\end{pf}

For clarity, we write $\bar m= \bar m(F)$ to display the
dependence of $\bar m$ in Lemma~\ref{obstH} on the nonlinear
operator $F$.

We are now ready to define the effective nonlinearity:
%e4.6 #&#
\begin{equation}
\label{Fbar} \overline F(0):= \sup \bigl\{ \alpha\in\mathbb{R} \dvtx  \bar m(F
- \alpha) > 0 \bigr\}.
\end{equation}
We extend this definition to all symmetric matrices in the obvious way.
For each $N\in{\mathbb{S}^d}$, we denote $F_N$ by
\[
F_N(M,y,\omega):= F(M+N,y,\omega)
\]
and then we set, for each $M\in{\mathbb{S}^d}$,
\[
\label{FbarM} \overline F(M):= \overline F_M(0).
\]
To check that $\overline F$ is well defined and finite, we first
observe that, by~(\ref{obstdn}) and~(\ref{obstup}),
\[
\inf_{y\in V} F(0,y,\omega) \geq0 \qquad\mbox{implies that } \mathcal{C}(V,\omega) = V
\]
and
\[
\sup_{y\in V} F(0,y,\omega) < 0 \qquad\mbox{implies that }
\mathcal{C}(V,\omega) = \varnothing.
\]
Using (F3) and the remarks in Section~\ref{PO}, it follows from these that
%e4.7 #&#
\begin{equation}
\label{easybnds} \essinf_{\omega\in\Omega} F(M,0,\omega) \leq\overline F(M) \leq
\esssup_{\omega\in\Omega} F(M,0,\omega).
\end{equation}
The monotonicity of the obstacle problem implies that $\alpha\mapsto
\bar m(F-\alpha)$ is a decreasing function, and thus $\bar m(F-\alpha) > 0$ for $\alpha< \overline F(0)$ and $\bar m(F-\alpha) = 0$ for $\alpha> \overline F(0)$.

It is immediate from the comparison principle for the obstacle problem
that, if $F_1$ and $F_2$ are two operators satisfying our hypotheses, then
%e4.8 #&#
\begin{equation}
\label{obvmono} \sup_{M\in{\mathbb{S}^d}} \esssup_{\omega\in\Omega} \bigl(
F_1(M,0,\omega) - F_2(M,0,\omega) \bigr)\leq0 \qquad\mbox{implies }\overline F_1 \leq\overline F_2.
\end{equation}
It is even more obvious that adding constants commutes with the
operation \mbox{$F \mapsto\overline F$}. From these facts, a number of
properties of $\overline F$ are immediate, the ones inherited from
uniform properties of $F$. A few of these are summarized in the
following lemma.

%
%le4.2 #&#
\begin{lem}\label{Fbarprop}
For every $M,N\in{\mathbb{S}^d}$ such that $M\leq N$, we have
%e4.9 #&#
\begin{equation}
\label{Fdegell} 0 \leq\overline F(M) - \overline F(N) \leq\Lambda\tr(N-M).
\end{equation}
Moreover, if $M\mapsto F(M,0,\omega)$ is positively homogeneous of
order one, odd or linear, then $\overline F$ possesses the same property.
\end{lem}

\begin{pf}
Each of the properties are proved using the comments before the
statement of the proposition. To prove~(\ref{Fdegell}), we simply
observe that, according to (F1), for all $(Y,y,\omega)\in{\mathbb
{S}^d}\times\mathbb R^d\times\Omega$,
%e4.10 #&#
\begin{equation}
F(M+Y,y,\omega) \leq F(N+Y,y,\omega) + \Lambda\tr(N-M)
\end{equation}
and then apply~(\ref{obvmono}). It is obvious that $\overline F$
inherits the properties of positive homogeneity and oddness from $F$,
and linearity follows from these.
\end{pf}

Observe that~(\ref{Fdegell}) asserts that $\overline F$ is degenerate
elliptic. If $F$ were uniformly elliptic, that is, $\lambda^{-1} \in
L^\infty(\Omega)$, then it follows from an argument nearly identical
to the one for~(\ref{Fdegell}) that $\overline F$ is uniformly
elliptic. For more general $\lambda^{-1} \in L^d(\Omega)$, the
operator $\overline F$ is uniformly elliptic as well, but the proof is
more complicated. We postpone it until the next subsection, since it is
convenient to deduce it as a consequence of Proposition~\ref
{parahomo}, which we prove first.

We next show that, in large domains, the contact set has nearly
constant density.

%
%le4.3 #&#
\begin{lem} \label{lspread}
For every~$\omega\in\Omega_2$ and $V,W \in\mathcal{L}$ with
$\overline W \subseteq V$,
%e4.11 #&#
\begin{equation}
\label{elimsubs} \lim_{t\to\infty} \frac{ |\mathcal{C}(tV,\omega) \cap tW|}{|tW|} = \bar m.
\end{equation}
\end{lem}
\begin{pf}
Let $U:= V\setminus W\in\mathcal{L}$ and fix $\omega\in\Omega_2$.
Observe that~(\ref{Cs-mono}) gives
%e4.12 #&#
\begin{equation}
\label{espup} \limsup_{t\to\infty} \frac{|\mathcal{C}(tV,\omega) \cap
tW|}{|tW|} \leq\lim
_{t\to\infty} \frac{|\mathcal{C}(tW,\omega
)|}{|tW|} = \bar m
\end{equation}
and, by the same argument,
\[
\limsup_{t\to\infty} \frac{|\mathcal{C}(tV,\omega) \cap
tU|}{|tU|} \leq\bar m.
\]
Therefore,
%
%e4.13 #&#
\begin{eqnarray}
\label{espdn}
\qquad\liminf_{t\to\infty} \frac{|\mathcal{C}(tV,\omega) \cap
tW|}{|tW|} &=&\liminf
_{t\to\infty} \frac{|\mathcal{C}(tV,\omega)
\cap tV|-|\mathcal{C}(tV,\omega) \cap tU|}{|tW|}
\nonumber\\[-8pt]\\[-8pt]
&\geq& \biggl( \frac{|V|}{|W|} - \frac{|U|}{|W|} \biggr) \bar m =
\bar m.\nonumber
\end{eqnarray}
Combining~(\ref{espup}) and~(\ref{espdn}) yields~(\ref{elimsubs}).
\end{pf}

\subsection*{Step two: Building approximate correctors}
The next step in the proof of Theorem~\ref{H} is to show that, in the
macroscopic limit, the obstacle problem controls the solution of the
Dirichlet problem
%e4.14 #&#
\begin{equation}
\label{free} \cases{F \bigl(D^2v,y,\omega \bigr) = 0, &\quad in
$V$,
\vspace*{3pt}\cr
v = 0, &\quad on $\partial V$.}
\end{equation}
As before, $V\in\mathcal{L}$ is a bounded Lipschitz domain and we
write $v=v(y,\omega; V,F)$.

The following proposition is the focus of this subsection.

%
%pr4.4 #&#
\begin{prop} \label{parahomo}
There exists an event $\Omega_3 \in\mathcal{F}$ of full
probability such that, for every $\omega\in\Omega_3$, $M\in
{\mathbb{S}^d}$ and $V\in\mathcal{L}$,
%e4.15 #&#
\begin{equation}
\label{parahomo-e} \lim_{t\to\infty} \frac{1}{t^2} \sup
_{y\in tV} \bigl\llvert v \bigl(y,\omega; tV,F_M-
\overline F(M) \bigr) \bigr\rrvert = 0.
\end{equation}
\end{prop}

Before we give its proof, we remark that Proposition~\ref{parahomo} is
a special case of Theorem~\ref{H}. We can see this by fixing $U\in
\mathcal{L}$, defining
\[
v^\varepsilon(x,\omega):= \varepsilon^2 v \biggl(
\frac{x}\varepsilon,  \omega;  \frac{1}\varepsilon U, F -
\overline F(0) \biggr)
\]
and then checking that $v^\varepsilon(\cdot,\omega)$ is the unique
solution of the boundary-value problem
%e4.16 #&#
\begin{equation}
\label{cell} \cases{ F \biggl(D^2v^\varepsilon,
\dfrac{x}\varepsilon,\omega \biggr) = \overline F(0), &\quad in $U$,
\vspace*{3pt}\cr
v^\varepsilon= 0, &\quad on $\partial U$.}
\end{equation}
The conclusion of Proposition~\ref{parahomo} then asserts that
%e4.17 #&#
\begin{equation}
v^\varepsilon\to0 \qquad\mbox{uniformly in } U\qquad\mbox{as } \varepsilon\to0,
\end{equation}
which is consistent with Theorem~\ref{H} since the zero function
$v\equiv0$ is obviously the unique solution
%e4.18 #&#
\begin{equation}
\label{dumb} \cases{ \overline F \bigl(D^2v \bigr) = \overline
F(0), &\quad in $U$,
\vspace*{3pt}\cr
v = 0, &\quad on $\partial U$.}
\end{equation}
As we show in the next subsection, Proposition~\ref{parahomo} actually
implies Theorem~\ref{H}. This is because, for large $R> 0$, the
function $\xi(y):= v (y,\omega;  B_R,  F_M-\overline F(M)
)$ is an ``approximate corrector'' in $B_R$ in the sense that it
satisfies the equation
%e4.19 #&#
\begin{equation}
F \bigl(M+D^2\xi,y,\omega \bigr) = \overline F(M) \qquad\mbox{in }B_R
\end{equation}
and is ``strictly subquadratic at infinity'' [i.e., satisfies~(\ref
{parahomo-e})]. This is precisely what is needed to implement the
perturbed test function method.

\begin{pf*}{Proof of Proposition~\ref{parahomo}}
According to the ergodic theorem, there exists an event $\Omega_4\in
\mathcal{F}$ of full probability such that, for every
$\omega\in\Omega_4$, $V \in\mathcal{L}$ and rational $q\in\mathbb
{Q}$ with $q>0$,
%e4.20 #&#
\begin{equation}
\label{Oksequp} \lim_{t \to\infty} \fint_{tV}
\lambda^{-d}(y,\omega) \,dy = \mathbb{E} \bigl[ \lambda^{-d}
\bigr]
\end{equation}
and
%e4.21 #&#
\begin{equation}
\label{Okseq} \lim_{t \to\infty} \fint_{tV}
\lambda^{-d}(y,\omega) \chi_{\{
\lambda< q \}}(y) \,dy = \mathbb{E} \bigl[
\lambda^{-d} \mathbh {1}_{\{ \lambda< q\}} \bigr].
\end{equation}
Note that, according to the ABP inequality (Proposition~\ref{pabp},
properly scaled), for every $\omega\in\Omega_4$,
%e4.22 #&#
\begin{equation}
\label{freeconte} \qquad\lim_{\alpha\to0}\limsup_{t\to\infty}
\frac{1}{t^2} \sup_{y\in
tV} \frac{1}{R^2} \bigl\llvert v
(y,\omega; tV,F ) - v (y,\omega; tV,F + \alpha ) \bigr\rrvert = 0.
\end{equation}
We now define $\Omega_3:= \Omega_2 \cap\Omega_4$, where
$\Omega_2$ is given in the statement of Lemma~\ref{obstH}.

We first show that, for all $\omega\in\Omega_3$, $V\in
\mathcal{L}$ and $M\in{\mathbb{S}^d}$
%e4.23 #&#
\begin{equation}
\label{freeupe} \liminf_{t\to\infty} \frac{1}{t^2} \inf
_{y\in tV} v \bigl(y,\omega; tV,F_M-\overline F(M)
\bigr) \geq0.
\end{equation}
We may assume that $M=0$ by replacing~$F$ with~$F_{-M}$ and that
$\overline F(0)=0$ by replacing~$F$ by~$F-\overline F(0)$. By~(\ref
{freeconte}), we may also suppose that $\bar m(F) = 0$ by
considering $F - \alpha$ for $\alpha> 0$ and then sending $\alpha\to
0$. Set
\[
K:= \esssup_{\omega\in\Omega} \bigl(F(0,0,\omega) \bigr)_+ = \esssup_{\omega
\in\Omega}
\sup_{y\in\mathbb R^d} \bigl(F(0,y,\omega) \bigr)_+.
\]
According to~(\ref{parahomo-e}) and~(\ref{obst-subs}), for every
$t>0$, the function $u:= w(\cdot,\omega; tV,F) - v(\cdot,\omega;tV,F)$ satisfies
\[
{\mathcal{P}}^-_{\lambda(y,\omega),\Lambda} \bigl(D^2u \bigr) \leq K \chi
_{\mathcal{C}(tV,\omega)} \qquad\mbox{in } tV
\]
and $u = 0$ on $\partial(tV)$. Using that $w\geq0$, the ABP
inequality (Proposition~\ref{pabp}, properly scaled) and~(\ref
{obstHe}), we obtain
%
%e4.24 #&#
\begin{eqnarray}
\label{edfgh}
&& \limsup_{t \to\infty} \frac{1}{t^2} \sup_{y\in tV} -\,v(y,\omega tV,F)\nonumber
\\
&&\qquad \leq \limsup_{t \to\infty}
\frac{1}{t^2} \sup_{y\in tV} u(y)
\\
&&\qquad \leq C K \limsup_{t\to\infty} \biggl( \fint_{tV} \lambda
^{-d}(y,\omega) \chi_{\mathcal{C}(tV,\omega)}(y) \,dy \biggr)^{1/d}.\nonumber
\end{eqnarray}
To estimate the integral on the right, we observe that, for each $k\in
\mathbb{N}$,
\[
\int_{tV} \lambda^{-d}(y,\omega)
\chi_{\mathcal{C}(tV,\omega)}(y) \,dy \leq \biggl( k^d\bigl|m(tV,\omega)\bigr| + \int
_{tV} \lambda ^{-d}(y,\omega) \chi_{ \{ \lambda< 1 /k \}}(y)
\,dy \biggr).
\]
Divide this by $|tV|$ and pass to the limit $t\to\infty$ using~(\ref
{Okseq}) to obtain
%e4.25 #&#
\begin{equation}
\qquad \limsup_{t\to\infty} \fint_{tV} \lambda^{-d}(y,
\omega) \chi _{\mathcal{C}(tV,\omega)}(y) \,dy \leq k^d\bar m(F) + \mathbb
{E} \bigl[ \lambda^{-d} \mathbh{1}_{\{ \lambda< 1/k\}} \bigr].
\end{equation}
Since $\bar m(F) = 0$, we may send $k \to\infty$ and combine the
resulting expression with~(\ref{edfgh}) to obtain~(\ref{freeupe}).

To complete the proof, we show that, for every $\omega\in\Omega_3$,
%e4.26 #&#
\begin{equation}
\label{freedowne} \limsup_{t\to\infty} \frac{1}{t^2} \sup
_{y\in tV} v \bigl(y,\omega; tV,F_M -\overline F(M)
\bigr) \leq0.
\end{equation}
As above, we may suppose that $M=0$ and $\overline F(0) = 0$. We may
also assume that $\bar m(F) > 0$, by considering $F + \alpha$ for
$\alpha> 0$ and then sending $\alpha\to0$, using~(\ref{freeconte}).
Since $v \leq w$, it suffices for~(\ref{freedowne}) to show that
%e4.27 #&#
\begin{equation}
\label{parabse} \limsup_{t\to\infty} \frac{1}{t^2}\sup
_{y\in tV} w (y,\omega; tV,F ) = 0.
\end{equation}
Furthermore, by the monotonicity of the obstacle problem it suffices to
show that
%e4.28 #&#
\begin{equation}
\label{parabse2} \limsup_{t\to\infty} \frac{1}{t^2}\sup
_{y\in tB_R} w (y,\omega; tB_{2R},F ) = 0,
\end{equation}
where $R> 1$ is large enough that $V \subseteq B_{R}$. Fix $r > 0$ and
observe that, by~Lemma~\ref{lspread} and an easy covering argument
using $\bar m(F) > 0$, there exists $T > 0$ sufficiently large
such that, for every $t\geq T$ and $x\in B_R$, the function $w(\cdot,
\omega;  tB_{2R}, F)$ vanishes at some point of $B(tx,tr)$. We
therefore have, for every $t\geq T$ and $x\in B_R$,
%e4.29 #&#
\begin{equation}
\label{cntcr} \frac{1}{t^2}\bigl|w(tx, \omega; tB_{2R}; F)\bigr| \leq
\mathop{\operatorname{osc}}_{B(tx,tr)} \frac{1}{t^2} w(\cdot, \omega;
tB_{2R}; F).
\end{equation}
We prove~(\ref{parabse2}) by showing that the lim-sup of the
right-hand side of~(\ref{cntcr}), as $t\to\infty$, is~$o(1)$
as~$r\to0$. For this, we rely on Proposition~\ref{poscillation}.

Notice that~(\ref{obst-subs}),~(\ref{Oksequp}) and the ABP inequality
(Proposition~\ref{pabp}) yield, for $t>0$ sufficiently large, the bound
%e4.30 #&#
\begin{equation}
\label{ctnctup} \frac{1}{t^2} \sup_{tB_{2R}} \bigl\llvert w(
\cdot,\omega;  tB_{2R},F) \bigr\rrvert \leq C K R^2,
\end{equation}
where $C$ depends only on $d$, $\Lambda$ and $\mathbb{E} [
\lambda^{-d}  ]$. Select $\mu> 0$ such that
\[
\mathbb{E} \bigl[ \lambda^{-d} \mathbh{1}_{\{\lambda< \mu\}
} \bigr] <
4^{-d} \delta,
\]
where $\delta> 0$, $\delta\in\mathbb{Q}$ is as in Proposition~\ref
{poscillation}. By~(\ref{Okseq}), and making $T > 0$ larger, if
necessary, we have that for all $t\geq T$, $r<r'<R$ and $x\in B_R$,
%e4.31 #&#
\begin{equation}
\label{covdn} \fint_{B_{tr'}(tx) } \lambda^{-d}(y,\omega)
\chi_{ \{ \lambda< \mu
\}} \,dy < \delta.
\end{equation}
To see this, consider a finite covering $\{ B_s(x_i) \}$ of $B_R$ by
balls of radius $s=2^{-k}R$ for some $k\in\mathbb{N}$. According
to~(\ref{Okseq}), for sufficiently large $t$, the average of $\lambda
^{-d}(y,\omega)\chi_{ \{ \lambda< \mu\}}$ in each of the balls
$B_{2s}(x_i)$ will be less than $4^{-d} \delta$. But every ball
$B_{r'}(x)$, with $s/2 \leq r' \leq s$ and $x\in B_R$, is contained in
one of the balls $B_{2s}(x_i)$. Since $4r' \geq s$, this yields
\[
\fint_{B_{tr'}(tx) } \lambda^{-d}(y,\omega)\chi_{ \{ \lambda< \mu
\}} \,dy
\leq4^{d} \fint_{B_{2s}(x_i)} \lambda^{-d}(y,\omega)\chi
_{ \{ \lambda< \mu\}} \,dy < \delta.
\]
Repeating this covering argument for $k=0,1,2,\ldots, \lceil\log
_2(R/r)  \rceil$ and making $T>0$ larger, if necessary, we
obtain~(\ref{covdn}) for every $t\geq T$, $r\leq r' \leq R$ and $x\in B_R$.

Iterating Proposition~\ref{poscillation}, using~(\ref
{ctnctup}),~(\ref{covdn}) as well as~(\ref{obst-sups}) and~(\ref
{obst-subs}), we obtain, for every $x\in B_R$ and $t\geq T$,
%e4.32 #&#
\begin{equation}
\mathop{\operatorname{osc}}_{B(tx,tr)} \frac{1}{t^2} w(\cdot, \omega;
tB_{2R}; F) \leq Cr^{\gamma}
\end{equation}
for some constants $\gamma>0$ and $C>0$ which may depend
on~$d$,~$\Lambda$, $\mathbb{E} [\lambda^{-d} ]$,~$\mu
$, $K$~and~$R$, but do not depend on $r$ or $T$. Combining this
with~(\ref{cntcr}) and sending $t\to\infty$ and then $r \to0$, we
obtain (\ref{parabse}), and thus the proposition.
\end{pf*}

We conclude the second part by showing that $\overline F$ is uniformly
elliptic and giving an estimate of its ellipticity. The proof is based
on Lemma~\ref{lcontact} and Proposition~\ref{parahomo}.

%
%pr4.5 #&#
\begin{prop} \label{ellipest}
There exists $c>0$, depending only on $d$ and $\Lambda$, such that
$\overline F$ is uniformly elliptic with constants $\lambda_0:= c
\mathbb{E}  [ \lambda^{-d}  ]^{-1}$ and $\Lambda$, that
is, for all $M,N\in{\mathbb{S}^d}$,
%e4.33 #&#
\begin{equation}
\label{unifell} {\mathcal{P}}^-_{\lambda_0,\Lambda}(M-N) \leq\overline F(M) -
\overline F(N) \leq{\mathcal{P}}^+_{\lambda_0,\Lambda} (M-N).
\end{equation}
\end{prop}
\begin{pf}
Select $M,N\in{\mathbb{S}^d}$ such that $M\geq N$. Fix $\omega\in
\Omega_3$ and define, for each $\varepsilon> 0$,
\begin{eqnarray*}
V_\varepsilon(x)&:=& \varepsilon^2v \biggl(\frac{x}\varepsilon,
\omega; \frac{1}\varepsilon B_1,F_M-\overline
F(M) \biggr) - \varepsilon ^2v \biggl(\frac{x}\varepsilon,
\omega; \frac{1}\varepsilon B_1,F_N-\overline
F(N) \biggr)
\\
&&{}+ \frac{1}2 x\cdot(M-N)x.
\end{eqnarray*}
It is easy to check that $V$ satisfies the inequality
%e4.34 #&#
\begin{equation}
{\mathcal{P}}^+_{\lambda(x/\varepsilon,\omega),\Lambda} \bigl( D^2V_\varepsilon \bigr)
\geq\overline F(M) - \overline F(N) \qquad\mbox{in } B_1.
\end{equation}
According to Proposition~\ref{parahomo},
%e4.35 #&#
\begin{equation}
\label{Veplim} V_\varepsilon(x) \rightarrow\tfrac{1}2 x\cdot(M-N)x
\qquad\mbox{as } \varepsilon\to0\qquad\mbox{uniformly in } B_1.
\end{equation}

Suppose that $M-N$ has a largest eigenvalue $a > 0$ with corresponding
normalized eigenvector $\xi\in\mathbb R^d$, $|\xi|=1$ so that
%e4.36 #&#
\begin{equation}
\label{bigeig} a \xi\otimes\xi\leq M-N \leq a I.
\end{equation}
Fix $\beta> 0$ and, for each $y\in\mathbb R^d$, denote by $\bar z(y)\in\mathbb R^d$ the (unique) point at which the map $x\mapsto\Phi
(x,y):= \frac{1}2 x\cdot(M-N)x + \beta|x-y|^2$ attains its (strict)
global minimum on $\mathbb R^d$. Note that
\[
\bar z(y) = (M-N + 2\beta I )^{-1} 2\beta y.
\]
In particular, $|\bar z(y)| \leq|y|$ and
\[
\bigl\llvert \xi\cdot\bar z(y) \bigr\rrvert = (a+2\beta)^{-1}
\bigl\llvert \xi\cdot(2\beta y) \bigr\rrvert \leq\frac{2\beta}{2\beta+ a} |y|.
\]

Applying~(\ref{Veplim}), we deduce that, for sufficiently small
$\varepsilon> 0$ and every~$y \in B_{1/3}$, the infimum in $B_1$ of
the map $x \mapsto V_\varepsilon(x) + \beta|x-y|^2$ is attained in
$B_{1/2}$ and any point $z$ at which the minimum is attained satisfies
%e4.37 #&#
\begin{equation}
\llvert z \cdot\xi\rrvert \leq\frac{2\beta}{2\beta+a}{\frac{1}3} <
\frac{\beta}{a}.
\end{equation}
Let $A:= \{ x \in B_{1/2}  \dvtx   |x\cdot\xi| < \beta/a \}$ and note
that $|A| \leq\beta/a$. In the case that $\overline F(M) - \overline
F(N) \geq- 2\beta$, we may apply Lemma \ref{lcontact}, using~(\ref
{Oksequp}), to obtain
\[
c \leq|B_{1/2}| \leq C \limsup_{\varepsilon\to0}\int
_{A} \lambda \biggl(\frac{x}\varepsilon,\omega
\biggr)^{-d} \,dx = C|A| \mathbb {E} \bigl[ \lambda^{-d} \bigr]
\leq C \frac{\beta}{a} \mathbb {E} \bigl[ \lambda^{-d} \bigr].
\]
This is impossible if $a \geq C\beta  \mathbb{E}  [ \lambda
^{-d}  ]$. Here, $C>0$ depends only on $d$ and $\Lambda$.

We conclude that $a /\beta\geq\widetilde C:=C \mathbb{E}  [
\lambda^{-d}  ]$ implies that $\overline F(M) - \overline F(N) <
- 2\beta$. Define $\lambda_0:= 2/ d\widetilde C$ and deduce that,
%
%Therefore, we may select $\lambda_0> 0$ small enough, depending on the
%appropriate quantities, so that by fixing $\beta:= \lambda_0 a\d/2$,
%the previous argument yields,
for all $M\geq N$,
\[
\overline F(M) - \overline F(N) \leq-2a/\widetilde C = -\lambda_0 a
d= {\mathcal{P}}^+_{\lambda_0,\Lambda}(aI) \leq{\mathcal {P}}^+_{\lambda_0,\Lambda} (M-N).
\]
Recalling~(\ref{Fdegell}), we also have, for every $M\geq N$,
\[
\overline F(M) - \overline F(N) \geq-\Lambda\tr(N-M) = {\mathcal
{P}}^-_{\lambda_0,\Lambda}(M-N).
\]
We have verified~(\ref{unifell}) for all $M,N\in{\mathbb{S}^d}$ with
$M\geq N$.

To remove the latter restriction, fix any $M,N\in{\mathbb{S}^d}$ and write
\begin{eqnarray*}
\overline F(M) - \overline F(N) &=& \overline F(M) - \overline F \bigl(M-(N-M)_-
\bigr)
\\
&&{}+ \overline F \bigl(M-(N-M)_- \bigr) - \overline F \bigl(M-(N-M)_-+(N-M)_+
\bigr)
\end{eqnarray*}
and observe by what we have shown above that
\[
\overline F(M) - \overline F(N) \leq{\mathcal{P}}^+_{\lambda
_0,\Lambda}
\bigl((N-M)_- \bigr) - {\mathcal{P}}^-_{\lambda_0,\Lambda} \bigl((N-M)_+ \bigr) = {
\mathcal{P}}^+_{\lambda_0,\Lambda}(M-N).
\]
This yields the second inequality of~(\ref{unifell}) and arguing again
after interchanging $M$ and $N$ yields the first inequality.
\end{pf}

\subsection*{Step three: Concluding by the perturbed test function method}
By adapting the classical perturbed test function method, first
introduced in the context of periodic homogenization by Evans~\cite
{E2}, we now complete the proof of Theorem~\ref{H}. The test functions
are perturbed by the approximate correctors constructed in
Proposition~\ref{parahomo}. The argument we present here is similar in
spirit to the one given in Section~4 of~\cite{CSW}, although a bit
less complicated.

\begin{pf*}{Proof of Theorem~\ref{H}}
Fix a bounded Lipschitz domain $U\in\mathcal{L}$, $g \in C(\partial
U)$ and an environment $\omega_0 \in\Omega_3$, where the event
$\Omega_3\in\mathcal{F}$ is given in the statement of
Proposition~\ref{parahomo}.

We first argue that, for every $x\in U$,
%e4.38 #&#
\begin{equation}
\label{Hup} \widetilde u(x):= \limsup_{\varepsilon\to0}
u^\varepsilon(x,\omega _0) \leq u(x).
\end{equation}

To show (\ref{Hup}), we begin by checking that $\widetilde u(x) \leq
g$ on $\partial U$. By approximation, we may assume that $g \equiv0$
and that $U$ is smooth (and in particular has the exterior ball
condition). By dilation, we may also assume that $F(0,\cdot,\omega)
\leq1$ and that $U \subseteq B_{R/2}(0)$. Given $y \in\partial U$, we
may select $B_r(x) \subseteq\mathbb{R}^d \setminus U$ such that
$\overline B_r(x) \cap\overline U = \{ y \}$. Given $h > 0$, we apply
Lemma~\ref{lbarrier} with the modification in Remark~\ref
{rbarrier}. Using $\omega_0 \in\Omega_3$, we may select
$\beta> 0$ and $r' \in(r, R-r)$ such that the solution $\varphi
^\varepsilon\in C(\overline B_R \setminus B_r)$ of
\[
\cases{ \mathcal{P}^-_{\lambda( x/\varepsilon,\omega_0),\Lambda} \bigl(D^2
\varphi^\varepsilon \bigr) = 1, &\quad in $B_R \setminus
B_r$,
\vspace*{3pt}\cr
\varphi^\varepsilon= \beta,&\quad on $\partial
B_R$,
\vspace*{3pt}\cr
\varphi^\varepsilon= 0, &\quad on $\partial
B_r$,}
\]
satisfies $\limsup_{\varepsilon\to0} \varphi^\varepsilon\leq h$ in
$V \cap B_{r'}(x)$. Since $U \subseteq B_R(x)$ and $u^\varepsilon\leq
0$ on $\partial U$, the comparison principle implies that
$u^\varepsilon\leq\varphi^\varepsilon$. It follows that
\[
\limsup_{\varepsilon\to0} \sup_{V \cap B_{r' - r}(y)} u^\varepsilon (
\cdot, \omega_0) \leq h.
\]
Since $h > 0$ was arbitrary, we conclude that $\widetilde u \leq g$ on
$\partial U$.

By the comparison principle, to prove~(\ref{Hup}) it suffices to check
that the function $\widetilde u(x):=\limsup_{\varepsilon\to0}
u^\varepsilon(x,\omega_0)$ satisfies, in the viscosity sense,
%e4.39 #&#
\begin{equation}
\label{tiluss} F \bigl(D^2\widetilde u \bigr) \leq0 \qquad\mbox{in }U.
\end{equation}
To verify~(\ref{tiluss}), we select a smooth test function $\phi\in
C^2(U)$ and a point $x_0\in U$ such that
\[
x \mapsto ( \widetilde u - \phi ) (x) \qquad\mbox{has a strict local maximum at }
x = x_0.
\]
We must show that $\overline F(D^2\phi(x_0)) \leq0$. Set $M:= D^2\phi
(x_0)$ and suppose on the contrary that $\theta:= \overline F(M) > 0$.

Since the local maximum of $\widetilde u-\phi$ at $x_0$ is strict,
there exists $r_0>0$ such that $B_{r_0}(x_0) \subseteq U$ and, for
every $0 < r \leq r_0$,
%e4.40 #&#
\begin{equation}
\label{stroom} (\widetilde u - \phi ) (x_0) > \sup
_{\partial B_r(x_0)} (\widetilde u - \phi ).
\end{equation}
We next introduce the perturbed test function
\[
\phi^\varepsilon(x):= \phi(x) + \varepsilon^2 v \biggl(
\frac{x}\varepsilon, \omega_0;  \frac{1}\varepsilon
B_{r_0}(x_0), F_M - \overline F(M) \biggr).
\]
We claim that, in some neighborhood of $x_0$, $\phi^\varepsilon$ is a
strict supersolution of the oscillatory equation at microscopic scale
$\varepsilon$. More precisely, we will argue that, for some suitably
small $0 < s < r_0$ to be selected below (and which may depend on~$\phi$),
%e4.41 #&#
\begin{equation}
\label{itpfm} F \biggl( D^2\phi^\varepsilon,
\frac{x}\varepsilon, \omega_0 \biggr) \geq\frac{1}2
\theta\qquad\mbox{in } B_s(x_0).
\end{equation}
To check~(\ref{itpfm}), we select a smooth test function $\psi\in
C^2(B_s(x_0))$ and a point $x_1 \in B_s(x_0)$ such that
\[
x \mapsto \bigl( \phi^\varepsilon- \psi \bigr) (x) \qquad\mbox{has a local
minimum at } x = x_1.
\]
Using the definition of $\phi^\varepsilon$ and rescaling, we have
%
%e4.42 #&#
\begin{eqnarray}
y &\mapsto& v \biggl( y,\omega_0; \frac{1}\varepsilon
B_{r_0}(x_0), F_M - \overline F(M) \biggr) -
\frac{1}{\varepsilon^2} \bigl(\psi (\varepsilon y) - \phi(\varepsilon y) \bigr)\nonumber
\\
\eqntext{\mbox{has a local minimum at } y = \dfrac{x_1}{\varepsilon}.}
\end{eqnarray}
Using the equation for $v$, we obtain
\[
F \biggl( M + D^2\psi(x_1) - D^2
\phi(x_1),  \frac{x_1}{\varepsilon
},  \omega_0 \biggr) -
\overline F(M) \geq0.
\]
Since $\phi\in C^2$, we may make $|M-D^2\phi(x_1)| = |D^2\phi
(x_0)-D^2\phi(x_1)|$ as small as we like by taking $s > 0$ small
enough. Thus, in light of (F2), we may fix $s> 0$ so that
\[
\biggl\llvert F \biggl( M + D^2\psi(x_1) -
D^2\phi(x_1),  \frac
{x_1}{\varepsilon},  \omega_0
\biggr) - F \biggl( D^2\psi(x_1), \frac{x_1}{\varepsilon},
\omega_0 \biggr) \biggr\rrvert \leq\frac{1}2\theta.
\]
The previous two inset inequalities and $\theta= \overline F(M)$ yield
%e4.43 #&#
\begin{equation}
F \biggl(D^2\psi(x_1), \frac{x_1}\varepsilon,
\omega_0 \biggr) \geq \frac{1}2\theta.
\end{equation}
This completes the proof of~(\ref{itpfm}).

An application of the comparison principle now yields
%e4.44 #&#
\begin{eqnarray}
\label{cpacntr} u^\varepsilon(x_0,\omega_0) -
\phi^\varepsilon(x_0) &\leq&\sup_{B_{s}(x_0)} \bigl(
u^\varepsilon(\cdot,\omega_0) - \phi ^\varepsilon \bigr)
\nonumber\\[-8pt]\\[-8pt]
& =&
\sup_{\partial B_{s}(x_0)} \bigl( u^\varepsilon(\cdot,\omega_0)
- \phi^\varepsilon \bigr).\nonumber
\end{eqnarray}
Taking the limsup of both sides of~(\ref{cpacntr}) as $\varepsilon\to
0$ and applying Proposition~\ref{parahomo}, we obtain
\[
\widetilde u(x_0) - \phi(x_0) \leq\sup
_{\partial B_s(x_0)} ( \widetilde u - \phi ).
\]
This contradicts~(\ref{stroom}) and completes the proof that
$\overline F(M) \leq0$, and hence of~(\ref{tiluss}), and hence
of~(\ref{Hup}).

It remains to show that, for every $x\in U$,
\[
\liminf_{\varepsilon\to0} u^\varepsilon(x,\omega_0) \geq
u(x).
\]
This is obtained by mimicking the argument above with very obvious
modifications. We omit the details.
\end{pf*}

%s5 #&#
\section{Breakdown of homogenization and regularity for $p<d$}\label{CE}
In this section, we show that the condition that the $d$th moment of
$\lambda^{-1}$ is finite is sharp for both the homogenization and
regularity results. We remark that the example we construct shows that
the exponent $p=d$ is sharp with respect to the general class of (fully
nonlinear) operators, but not with respect to the subclass of linear
operators. One interpretation of the reason for this difference is that
some nonlinear equations correspond to stochastic optimal control
problems, and the controller is under no obligation to select a
stationary control. A variant of our construction leads to a linear
counterexample for all $p<1$, which was already discovered in~\cite
{GZ} (see also~\cite{BD}) using a similar trap model. The range $1\leq
p <d$ thus remains open in the linear case; we believe that $p=1$ is
the critical exponent.

For each $p<d$, we construct a stationary-ergodic random environment
$(\Omega,\mathcal{F},\mathbb{P},\tau)$ and stationary random field
$\lambda\dvtx \mathbb R^d\times\Omega\to(0,1]$ such that
%e5.1 #&#
\begin{equation}
\label{momcex} \mathbb{E} \bigl[ \lambda^{-p} \bigr] < +\infty,
\end{equation}
but for which homogenization fails for the equation
%e5.2 #&#
\begin{equation}
{\mathcal{P}}^-_{\lambda ( x/\varepsilon,\omega),1} \bigl(D^2u^\varepsilon \bigr) =
1.
\end{equation}
To show the breakdown of homogenization, we check that the solution
$u^\varepsilon$ of the Dirichlet problem
%e5.3 #&#
\begin{equation}
\label{counterdp} \cases{{\mathcal{P}}^-_{\lambda ( x/,\omega),1} \bigl(D^2u^\varepsilon
\bigr) = 1, &\quad in $B_1$,
\vspace*{3pt}\cr
u^\varepsilon= 0, &\quad on $
\partial B_1$,}
\end{equation}
satisfies $\lim_{\varepsilon\to0} u^\varepsilon(0,\omega) =
+\infty$ almost surely. We conclude that there is no ``effective'' ABP
inequality, in the limit $\varepsilon\to0$, and hence no effective
regularity or effective equation. The random field $\lambda\dvtx \mathbb
R^d\times\Omega\to(0,1]$ we construct has a finite range of
dependence, so even this strongest possible mixing assumption cannot
save homogenization for a general nonlinear operator without a bounded
$d$th moment of ellipticity.

The idea underlying the construction of $\lambda$ is to build spatial
``traps'' where, from the probabilistic perspective, the corresponding
controlled diffusion process becomes stuck for long periods of time,
resulting in subdiffusive behavior on large scales. We fix $0 < \alpha
<1$ small, take $0 < \lambda_* < 1/2d$ to be selected below and
choose, for each $k\in\mathbb{N}$, a random arrangement $P_k(\omega
)\subseteq\mathbb R^d$ of points (also specified below). We construct
the random field $\lambda$ in such a way that $0 < \lambda\leq
\lambda_*$ almost surely and $\lambda(y,\omega)\leq\lambda_k:= 1/
(k^{1+\alpha} \log^3 (2+k))$ in each ball of radius $1$ with center
in $P_k(\omega)$. To be more precise, for each $k\in\mathbb{N}$ we
select a (deterministic) continuous function $\theta_{k}$ on $\mathbb
R^d$ which is at most $\lambda_k$ in $B_1$, takes the value $\lambda
_*$ in $\mathbb R^d\setminus B_2$ and satisfies $\lambda_k \leq\theta
_{k}(y)$ and $\theta_k(y) \leq\lambda_*$. We then set
\[
\lambda(y,\omega): = \inf_{k\in\mathbb{N}} \inf_{x\in P_k(\omega
)}
\theta_k(y-x).
\]
We may also easily arrange that the family $\{ \theta_k \}_{k\in
\mathbb{N}}$ is equicontinuous.

We take the point configurations $P_k$ to be independent Poisson point
processes (cf.~\cite{DV}), with intensities depending on $k$ such
that the expected number of points of $P_k \cap V$ is equal to $a|V|
k^{-1-d}$, where $a>0$ is a parameter independent of $k$ which we also
choose below. Since the series $\sum_{k=1}^\infty k^{-1-d}$ converges,
it follows that the number of points of $\bigcup_{k=1}^\infty P_k$ is
almost surely locally finite by the Borel--Cantelli lemma, and this
implies that $\lambda(0,\omega) > 0$ almost surely. In fact, for all
$p<d/ (1+ \alpha)$,
\begin{eqnarray*}
\mathbb{E} \bigl[ \lambda^{-p} \bigr] &\leq& 1 + Ca \sum
_{k=1}^\infty\lambda_k^{-(1+\alpha) p}
k^{-1-d}
\\
&=& 1 + Ca\sum_{k=1}^\infty
k^{-1-d+(1+\alpha) p} \log^{3p} (2+k) < \infty.
\end{eqnarray*}
The stationarity of the Poisson point processes implies that $\lambda$
is a stationary function, and it is clear that $\lambda(\cdot,\omega
)$ is uniformly continuous (almost surely in $\omega$) since the
family $\{ \theta_k \}_{k\in\mathbb{N}}$ is equicontinuous.

Let us see how we can increase the frequency of ``traps'' (i.e., regions
in which~$\lambda$ is small) by taking~$a>0$ large. For each fixed $t
> 1+a \log t |V|$, we see that
%
%e5.4 #&#
\begin{eqnarray}\label{percprob}
&& \mathbb{P} \bigl[ P_k \cap t (\log t)^{1/d}
V \neq\varnothing \mbox{ for some } k \geq t \bigr]
\nonumber\\[-8pt]\\[-8pt]
&&\qquad \geq 1 - \prod_{k=\lceil t \rceil
}^\infty \bigl( 1 - a
t^d \log t |V| k^{-1-d} \bigr) \geq1 - \exp \bigl(-a\log t |V|\bigr).\nonumber
\end{eqnarray}
Choose the constant $a>0$ large enough that $a|B_1| > d+1$, so that
\[
\mathbb{P} \bigl[ P_k \cap t (\log t)^{1/d} V \neq
\varnothing \mbox{ for some } k \geq t \bigr] \geq1 - t^{-(d+1)}.
\]
By covering $B_{t^2}$ with $Ct^d$ balls of radius $\frac{1}6 t(\log
t)^{1/d}$, we deduce that
\[
\mathbb{P} \bigl[ \mbox{for all } x\in B_{t^2}, \dist(x,P_k)
< \tfrac{1}3 t (\log t)^{1/d} \mbox{ for some } k \geq t \bigr]
\geq 1- Ct^{-1}.
\]
By using Borel--Cantelli along the sequence $t_j=2^{j}$, it follows that
%e5.5 #&#
\begin{equation}
\label{eBC} \qquad\mathbb{P} \bigl[ \exists s > 1\mbox{ s.t. } \forall t > s, x\in
B_{t^2}, \exists k \geq t \mbox{ s.t. } \dist(x,P_k) < t (
\log t)^{1/d} \bigr] = 1.
\end{equation}

The previous line says that we will have sufficiently many traps to
work with. We next measure the local effect of one trap.

%
%le5.1 #&#
\begin{lem} \label{lcounterbend}
Fix $\omega\in\Omega$, and suppose that $k\geq t >10$, $\dist
(0,P_k(\omega)) < t (\log t)^{1/d}$, $R(t):= 10t (\log t)^{1/d}$ and
$v\in C(B_R)$ satisfies
\[
\label{ederp1} \cases{{\mathcal{P}}^-_{\lambda ( x,\omega ),1} \bigl(D^2v
\bigr) \geq1, &\quad in $B_{R(t)}$,
\vspace*{3pt}\cr
v \geq\ell,&\quad on $\partial
B_{R(t)}$,}
\]
where $\ell$ is an affine function. If $0 < \lambda_* < 1/2d$ is
chosen small enough (depending on $\alpha$), then there exists $c,q>0$
depending on $d$ and $\lambda_*$, but not on $t$, such that
%e5.6 #&#
\begin{equation}
\label{eCEbump} v(0) \geq\ell(0) + cR(t)^2 \bigl(\log(t)
\bigr)^q
\end{equation}
for a constant $c>0$ depending only on $d$ and a lower bound for
$\alpha$.
\end{lem}
\begin{pf}
We may assume with loss of generality that $\ell=0$. The goal is to
find an explicit subsolution, taking advantage of the trap near the
origin and the fact that the ellipticity is larger than $\lambda
_*^{-1}$. Set $\beta:= 1-2d\lambda_* >0$, $a = (1-\beta)/2 =
d\lambda_* >0$ and
\[
\phi(x):= - \bigl( a + |x|^2 \bigr)^{\beta/2}.
\]
The Hessian of $\phi$ is given by
\[
D^2\phi(x) = \beta \bigl( a +|x|^2 \bigr)^{\beta/2-2}
\biggl( \bigl((1-\beta)|x|^2-a \bigr) \frac{x\otimes x}{|x|^2} - \bigl(a +
|x|^2 \bigr) \biggl(I- \frac{x\otimes x}{|x|^2} \biggr) \biggr).
\]
For $\mu\leq\lambda_*$, we find that
%
%e5.7 #&#
\begin{eqnarray}
{\mathcal{P}}^-_{\mu,1} \bigl(D^2\phi(x) \bigr) \leq\beta
\bigl( a + |x|^2 \bigr)^{\beta/2-2} \bigl( (\beta-1+a) +(a+1)
\mu(d-1) \bigr) \leq0\nonumber
\\
\eqntext{\mbox{in } \mathbb R^d\setminus B_1}
\end{eqnarray}
and, for a constant $C>0$ depending only on $d$,
\[
{\mathcal{P}}^-_{\mu,1} \bigl(D^2\phi(x) \bigr) \leq C\mu
\bigl( a + |x|^2 \bigr)^{\beta/2-2} \leq C\mu\lambda_*^{\beta/2-2}
\qquad\mbox{in } B_1.
\]

Now suppose $x_0\in P_k(\omega)$ such that $|x_0| < t (\log t)^{1/d}$.
Then for a small constant~$c>0$ depending only on dimension, the function
\[
\psi(x):= c t^{1+\alpha} (\log t)^3 \lambda_*^{2-\beta/2}
\phi(x-x_0)
\]
satisfies
\[
{\mathcal{P}}^+_{\lambda(x,\omega),1} \bigl(D^2\psi \bigr) \leq1 \qquad\mbox{in } \mathbb R^d.
\]
By the comparison principle,
\[
\psi(0) - v(0) \leq\max_{\partial B_R} (\psi-v) = \max
_{\partial
B_R} \psi,
\]
which yields, for $T:= t(\log t)^{1/d}$,
\begin{eqnarray*}
v(0) & \geq&\psi(0) - \max_{\partial B_R} \psi
\\
& \geq& c t^{1+\alpha} (\log t)^3 \lambda_*^{2-\beta/2} \bigl(
- \bigl(a+T^2 \bigr)^{\beta/2} + \bigl(a+(R-1)^2
\bigr)^{\beta/2} \bigr)
\\
& \geq& c t^{1+\alpha} (\log t)^3 \lambda_*^{2-\beta/2}
T^{\beta} \qquad\qquad\mbox{(since $R > 10T > 100a$)}
\\
& =& \bigl( c\lambda_*^{2-\beta/2} \bigr) T^{1+\alpha+\beta} (\log
t)^q,
\end{eqnarray*}
where $q:= 3-(1+\alpha)/2 > 0$. By taking $\lambda_*:= \alpha/2d$
so that $\alpha+\beta+1=2$, and noting that $R(t) = 10T$, we
obtain~(\ref{eCEbump}).
\end{pf}

The above lemma, after rescaling and in light of~(\ref{eBC}), implies
that the difference of $u^\varepsilon$ and the paraboloid $c(\log t)^q
(1-|x|^2)$, with $\varepsilon=t^{-2}$, cannot achieve its maximum on
$B_1$ except in a boundary strip of $\partial B_1$ of width at most
$t^{-1} = \sqrt\varepsilon$. This easily gives that $u^\varepsilon
\to+\infty$ locally uniformly in $B_1$ with at least rate $|\log
\varepsilon|^q$.

%sA #&#
\begin{appendix}\label{OO}
\section*{Appendix: Elementary properties of the~obstacle~problem}
For the convenience of the reader, we briefly review (and sketch the
proofs of) some well-known properties of the obstacle problem
%eA.1 #&#
\begin{equation}
\label{obst-a} \cases{\min \bigl\{ F \bigl(D^2w,y \bigr),  w \bigr\}
= 0, &\quad in $V$,
\vspace*{3pt}\cr
w = 0, &\quad on $\partial V$.}
\end{equation}
We have dropped the dependence of $F$ on $\omega$ since the random
environment plays no role here, and we furthermore assume that $F$ is
uniformly elliptic by the remarks preceding Theorem~\ref{H}.

First of all, problem~(\ref{obst-a}) has a unique solution $w\in
C(\overline V)$, which may be expressed as the least nonnegative
supersolution of $F=0$ in $V$:
%eA.2 #&#
\begin{equation}
\label{obstdn} w(x; V) = \inf \bigl\{ u(x) \dvtx  u \geq0\mbox{ in } V\mbox{ and } F
\bigl(D^2u,y \bigr) \geq0 \mbox{ in } V \bigr\}.
\end{equation}
This from the Perron method and the fact that the obstacle problem has
a comparison principle. Immediate from this expression is that $w$ is a
global subsolution:
%eA.3 #&#
\begin{equation}
\label{obst-sups} F \bigl(D^2w,y \bigr) \geq0 \qquad\mbox{in } V
\end{equation}
as well as the monotonicity property:
%eA.4 #&#
\begin{equation}
\label{obst-mono} V \subseteq W \qquad\mbox{implies that } w(\cdot;  V) \leq
w(\cdot; W)\qquad\mbox{on } \overline V.
\end{equation}
In order to use some regularity theory, we also need the fact that $w$ satisfies
%eA.5 #&#
\begin{equation}
\label{obst-subs} F \bigl(D^2w,y \bigr) \leq k \chi_{\{ w=0\}}
\qquad\mbox{in } V\qquad\mbox{where }k:= \sup_{y\in V}
\bigl(F(0,y) \bigr)_+.
\end{equation}
This is typically handled by considering an approximate equation with a
penalty term (the Levy--Stampacchia penalization method) whose
solutions satisfy~(\ref{obst-subs}) and showing that their uniform
limit is $w$. We instead opt for a more natural and simpler proof by
showing that $w$ is also given by the formula
%eA.6 #&#
\begin{equation}
\label{obstup} \qquad w(x; V) = \sup \bigl\{ u(x) \dvtx  u \leq0 \mbox{ on } \partial V
\mbox{ and } F \bigl(D^2u,y \bigr) \leq k \chi_{\{ u \leq0\}} \mbox{ in }
V \bigr\}.
\end{equation}
It is clear that~(\ref{obstup}) implies~(\ref{obst-subs}). To check
the former, we let $\hat w$ denote the expression on the right-hand
side and observe that, since the zero function belongs to the
admissible class by the definition of $k$, we have $\hat w \geq0$. The
Perron method implies that $\hat w$ satisfies $F(D^2\hat w,y) = 0$ in
$\{ \hat w > 0\}$. Therefore, $\hat w$ is a solution of~(\ref{obst}),
and by uniqueness we deduce $w=\hat w$.
\end{appendix}

% zodis "Acknowledgments" paliekamas pagal autoriu
\section*{Acknowledgments}
We thank Xavier Cabr\'e, Robert Kohn and Luis Silvestre for helpful discussions and Ofer
Zeitouni for bringing~\cite{BD} to our attention.

%suskaldyti doi

% imsref loaded by linak, 2014-01-10 09:36:42
% imsref loaded by linak, 2014-01-10 09:47:06
% imsref loaded by linak, 2014-07-29 11:08:29
% imsref loaded by linak, 2014-07-29 11:13:26

\printaddresses

\end{document}